\documentclass[oneside]{amsart}

\usepackage[latin1]{inputenc}
\usepackage{amsmath}
\usepackage{amssymb}
\usepackage{graphicx}
\usepackage{color}
\usepackage[all]{xy}

\newcommand{\MS}{{\medskip}}
\newcommand{\BS}{{\bigskip}}
\newcommand{\NI}{{\noindent}}

\newcommand{\N}{\mathbb{N}}
\newcommand{\Z}{\mathbb{Z}}
\newcommand{\Q}{\mathbb{Q}}
\newcommand{\R}{\mathbb{R}}
\newcommand{\C}{\mathbb{C}}
\newcommand{\F}{\mathbb{F}}

\newcommand{\Mm}{\mathcal{M}}
\newcommand{\Nn}{\mathcal{N}}
\newcommand{\Oo}{\mathcal{O}}

\newcommand{\Ii}{\mathcal{I}}
\newcommand{\jj}{\mathcal{J}}

\newcommand{\om}{\omega}

\newcommand{\omui}{\omega_{\mu}^{i}}

\newcommand{\si}{\sigma}

\newcommand{\tG}{\widetilde G}
\newcommand{\tT}{\widetilde T}

\newcommand{\jjui}{\jj_\mu^{i}}
\newcommand{\jjuo}{\jj_\mu^{0}}

\newcommand{\tjj}{\widetilde{\jj}}

\newcommand{\tjjuco}{\tjj^0_{\mu,c}}

\newcommand{\tjjuci}{\tjj^i_{\mu,c}}

\newcommand{\tIi}{\widetilde{\Ii}}

\newcommand{\ccrit}{c_{\text{crit}}}
\newcommand{\tR}{\widetilde{R}}
\newcommand{\tpsi}{\widetilde{\psi}}
\newcommand{\tpsis}{\widetilde{\psi}^{*}}

\newcommand{\id}{\mathrm{id}}

\newcommand{\Tor}{\operatorname{Tor}}
\newcommand{\SFr}{\mathrm{SFr}}
\newcommand{\rk}{\operatorname{rk}}
\newcommand{\hocolim}{\operatorname{hocolim}}

\newcommand{\GL}{\mathrm{GL}}

\newcommand{\U}{\mathrm{U}}
\newcommand{\SO}{\mathrm{SO}}
\newcommand{\SU}{\mathrm{SU}}
\newcommand{\PSL}{\mathrm{PSL}}
\newcommand{\BSO}{\mathrm{BSO}}
\newcommand{\BSU}{\mathrm{BSU}}

\newcommand{\Symp}{\mathrm{Symp}}
\newcommand{\Diff}{\mathrm{Diff}}
\newcommand{\BSymp}{\mathrm{BSymp}}

\newcommand{\Iso}{\mathrm{Iso}}
\newcommand{\Hol}{\mathrm{Hol}}
\newcommand{\Map}{\mathrm{Map}}
\DeclareMathOperator{\sym}{sym}

\newcommand{\FDiff}{\mathrm{FDiff}}
\newcommand{\BFDiff}{\mathrm{BFDiff}}

\newcommand{\Gui}{G_{\mu}^{i}}
\newcommand{\Guo}{G_{\mu}^{0}}
\newcommand{\Gul}{G_{\mu}^{1}}

\newcommand{\tGuc}{\widetilde{G}_{\mu,c}}
\newcommand{\tGuco}{\widetilde{G}^{0}_{\mu,c}}
\newcommand{\tGucl}{\widetilde{G}_{\mu,c}^1}
\newcommand{\tGuci}{\widetilde{G}_{\mu,c}^i}

\newcommand{\BK}{BK}
\newcommand{\BGu}{BG_{\mu}}
\newcommand{\BGui}{BG_{\mu}^{i}}
\newcommand{\tBG}{B\widetilde{G}}

\newcommand{\tBGuc}{B\widetilde{G}_{\mu,c}}
\newcommand{\tBGuco}{B\widetilde{G}^{0}_{\mu,c}}

\newcommand{\Emb}{\mathrm{Emb}}
\newcommand{\IEmb}{\Im\mathrm{Emb}}
\newcommand{\Embuci}{\Emb_{\om}^{i}(c,\mu)}
\newcommand{\IEmbuci}{\IEmb_{\om}^{i}(c,\mu)}

\newcommand{\IEmbuco}{\IEmb_{\om}^{0}(c,\mu)}

\newcommand{\IEmbucl}{\IEmb_{\om}^{1}(c,\mu)}

\newcommand{\IEmbuc}{\IEmb_{\om}(c,\mu)}

\newcommand{\CP}{\mathbb{C}P}

\newcommand{\STS}{S^2\times S^2}
\newcommand{\PbP}{\CP^2\#\,\overline{\CP}\,\!^2}
\newcommand{\NTB}{S^{2}\widetilde{\times} S^{2}}
\newcommand{\tJ}{\widetilde{J}}

\newcommand{\tM}{\widetilde{M}}
\newcommand{\tMuc}{\tM_{\mu,c}}
\newcommand{\Muo}{M^0_{\mu}}
\newcommand{\Mul}{M^1_{\mu}}
\newcommand{\Mui}{M^i_{\mu}}
\newcommand{\tMuco}{\widetilde{M}^0_{\mu,c}}

\newcommand{\tMuci}{\widetilde{M}^i_{\mu,c}}

\newcommand{\from}{\leftarrow}
\newcommand{\into}{\hookrightarrow}

\theoremstyle{plain}
\newtheorem{thm}{Theorem}[section]
\newtheorem*{thm*}{Theorem}
\newtheorem{prop}[thm]{Proposition}
\newtheorem*{prop*}{Proposition}
\newtheorem{lemma}[thm]{Lemma}
\newtheorem*{lemma*}{Lemma}
\newtheorem{cor}[thm]{Corollary}
\newtheorem*{cor*}{Corollary}

\newtheorem*{conj*}{Conjecture}

\theoremstyle{definition}

\newtheorem*{defn*}{Definition}

\theoremstyle{remark}
\newtheorem{remark}[thm]{Remark}
\newtheorem*{remark*}{Remark}

\newtheorem*{remarks*}{Remarks}

\definecolor{vert}{rgb}{0,0.5,0}

\begin{document}

\title[Space of symplectic balls]{The homotopy type of the space of symplectic balls in rational ruled $4$-manifolds}

\author[S. Anjos]{S\'ilvia Anjos}
\address{Centro de Análise Matemática, Geometria e Sistemas Dinâmicos \\ Departamento de Matemática \\  Instituto Superior T\'ecnico \\  Lisboa \\ Portugal}
\email{sanjos@math.ist.utl.pt}
\thanks{Partially supported by FCT/POCTI/FEDER and by project POCTI/2004/MAT/57888}

\author[F. Lalonde]{Fran\c{c}ois Lalonde}
\address{Universit\'e de Montr\'eal \\ Montr\'eal \\ Canada}
\email{lalonde@dms.umontreal.ca}
\thanks{Partially supported by NSERC grant OGP 0092913 and FQRNT grant  ER-1199.}

\author[M. Pinsonnault]{Martin Pinsonnault}
\address{The University of Western Ontario \\ London \\ Canada}
\email{mpinson@uwo.ca}

\date{October 2008}

\begin{abstract}
Let $M:=(M^{4},\om)$ be a $4$-dimensional rational ruled symplectic manifold and denote by $w_{M}$ its Gromov width. Let $\Emb_{\omega}(B^{4}(c),M)$ be the space of symplectic embeddings of the standard ball of radius $r$, $B^4(c) \subset \R^4$ (parametrized by its capacity $c:= \pi r^2$), into $(M,\om)$. By the work of Lalonde and Pinsonnault \cite{LP}, we know that there exists a critical capacity $\ccrit\in (0,w_{M}]$ such that, for all $c\in(0,\ccrit)$, the embedding space $\Emb_{\omega}(B^{4}(c),M)$ is homotopy equivalent to the space of symplectic frames $\SFr(M)$. We also know that the homotopy type of $\Emb_{\omega}(B^{4}(c),M)$ changes when $c$ reaches $\ccrit$ and that it remains constant for all $c\in[\ccrit,w_{M})$. In this paper, we compute the rational  homotopy type, the minimal model, and the cohomology with rational coefficients  of $\Emb_{\omega}(B^{4}(c),M)$ in the remaining case $c\in[\ccrit,w_{M})$. In particular, we show that it does not have the homotopy type of a finite CW-complex. Some of the key points in the argument are the calculation of the rational homotopy type of the classifying space of the symplectomorphism group of the blow up of $M$, its comparison with the group corresponding to $M$, and the proof that the space of compatible integrable complex structures on the blow up is weakly contractible. 
\end{abstract}

\maketitle

\section{Introduction}
 
We compute in this paper the rational homotopy type, the minimal model, and the cohomology with rational coefficients of the space of  embedded symplectic balls of capacity $c$  in any closed rational ruled $4$-manifold. We consider only minimal ruled manifolds in the sense that they are not blow-ups of ruled manifolds. By the classification theorem for rational ruled $4$-manifolds \cite{LM}, any such manifold is symplectomorphic, after rescalling, to either
\begin{itemize}
\item the topologically trivial $S^2$-bundle over $S^2$,  $M^0_\mu=(S^2 \times S^2, \omega^0_\mu)$, where $\omega^0_\mu$ is the split symplectic form $\om (\mu) \oplus \om(1)$  with area $\mu \geq 1$ for the first $S^2$-factor, and with area $1$ for the second factor; or 
\item the topologically non-trivial $S^2$-bundle over $S^2$, $M^1_\mu=(S^2 \widetilde{\times} S^2, \omega^1_\mu)$, diffeomorphic to \\ $\PbP$ equi\-pped with the standard K\"ahler form $\omega^1_\mu$ where the symplectic area of the exceptional divisor is $\mu >0$ and the area of a projective line is $\mu+1$ (this implies that the area of the fiber is $1$).  
\end{itemize} 
Note that the second bundle is, topologically, the only non-trivial $S^2$-bundle over $S^2$. Let $B^4(c)\subset\R^4$ be the closed standard ball of radius $r$ and capacity $c=\pi r^2$ equipped with the restriction of the symplectic structure $\om_{st}=dx_1\wedge dy_1+dx_2\wedge dy_2$ of $\R^4$. Let $\Embuci$ be the space, endowed with the $C^{\infty}$-topology, of all symplectic embeddings of $B^4(c)$ in $M^i_{\mu}$. Finally, let $\IEmbuci$ be the space of subsets of $M^i_\mu$ that are images of maps belonging to $\Embuci$ defined as the topological quotient
\begin{equation} \label{FibrationEmb-IEmb}
\Symp(B^4 (c))\hookrightarrow \Embuci
\longrightarrow\IEmbuci
\end{equation}
where $\Symp(B^4(c))$ is the group, endowed with the $C^{\infty}$-topology, of
symplectic diffeomorphisms of the closed ball, with no restrictions on the
behavior on the boundary (thus each such map extends to a symplectic diffeomorphism of a neighborhood of $B^4(c)$ that sends $B^4(c)$ to itself). We may view $\IEmbuci$ as the space of all unparametrized balls of capacity $c$ in $M^i_\mu$.

\subsection{Preliminary results}

Recall that the Non-squeezing theorem implies that $\Embuci$ is empty for $c \geq 1$; it is then easy to see that the Gromov width of all spaces $M^i_\mu$ is equal to $1$ and that, actually, the space $\Embuci$ is non-empty if and only if $c \in (0,1)$.  It was proved in~\cite{LP} Corollary~1.2 and in~\cite{Pi} Corollary~1.9 that the homotopy type of $\Embuci$ can be completely understood for some special values of $\mu$, namely

\begin{prop}\label{prop:RetractionExceptionalCases}
Let $\phi:\Embuci\to\SFr(\Mui)$ be the map that associates to an embedding $\iota:B^{4}(c)\into\Mui$ the symplectic frame at the origin.
\begin{enumerate}
\item For $\mu=1$ and $i=0$, that is, for $S^{2}\times S^{2}$ with factors of equal area, the map $\phi$ is a homotopy equivalence for all values $c\in(0,1)$. Consequently, the space of unparametrized balls $\IEmbuco$ is homotopy equivalent to $S^{2}\times S^{2}$.
\item In the twisted case, for any $\mu$ in the range $(0,1]$, the map $\phi$ is a homotopy equivalence for all values $c\in(0,1)$. Hence, the space $\IEmbucl$ is homotopy equivalent to $M^1_\mu$ for these values of $\mu$. 
\end{enumerate}
\end{prop}

 We will therefore assume in this article that $\mu > 1$. Denote by $\ell$ the ``low integral part'' of $\mu$, i.e the largest integer strictly smaller than $\mu$. Using an inflation argument, it was shown in Lalonde-Pinsonnault \cite{LP} Theorem 1.7 and in Pinsonnault \cite{Pi} Theorem 1.7 that:

\begin{prop} \label{prop:RetractionSmallCase}
The space $\Embuci$ is homotopy equivalent to the space of symplectic frames of $M^i_\mu$ for all values $c\in (0, \mu - \ell)$. Hence, in this range of $c$'s, the space $\IEmbuci$ is homotopy equivalent to the manifold $M^i_\mu$ itself. Moreover, the homotopy type of $\IEmbuci$ changes when $c$ reaches the critical capacity $\mu - \ell$ and remains constant for all $c\in[\mu - \ell,1)$. 
\end{prop}

  Define the {\em critical capacity} $\ccrit\in (0,1]$ by setting $\ccrit:=\mu-\ell$.  In this paper, we will therefore restrict our attention to the remaining cases, namely to the values $\mu > 1$ and $c\geq \ccrit$ in both the split and non-split bundles.   

\subsection{The general framework}

Let $\Mui$ be a normalized rational ruled $4$-manifold with $\mu > 1$ and consider $c \in [\ccrit,1)$. The main results of this paper are:

\begin{itemize}
\item Theorem~\ref{thm:homotopy-type} that gives the rational homotopy type of $\Im \Embuci$, expressed as a non-trivial fibration whose base and fiber are explicitely given,
\item Theorem~\ref{thm:MinimalModelIEmb} that computes the minimal model of $\Im \Embuci$, and 
\item Corollary~\ref{cor:RationalCohomologyRingIEmb}
 that computes the rational cohomology ring of $\Im \Embuci$.
\end{itemize}
 
 In particular, these results imply that if $c \in [\ccrit,1)$, then the topological space $\Im \Embuci$ does not have the homotopy type of a finite dimensional $CW$--complex. 

In order to obtain the previous results we need two fundamental calculations, namely:

\begin{itemize}
\item the computation of the rational homotopy type of  $B\Symp(\tMuci)$, the classifying space of the symplectomorphism group of the blow-up of $M^i_\mu$ at a ball of capacity $c$  (Theorem~\ref{thm:HomotopyPushoutBlowUp}), as well as its rational cohomology (Theorem~\ref{thm:RationalCohomologyAlgebra}), and
\item the calculation of the structure of the space of compatible integrable complex structures on the blow-up of $M^i_\mu$, and in particular the fact that this space is weakly contractible (Appendix ~\ref{se:appendixA}).
\end{itemize}
\MS

    Here is a brief description of the approach to prove these results.
McDuff showed in \cite{MD:Isotopie} that the space $\IEmbuci$ is path-connected. By extension of Hamiltonian isotopies, one sees immediatly that the group of Hamiltonian diffeomorphisms of $\Mui$ acts transitively on $\IEmbuci$. Note that under the restriction $\mu > 1$, the group of Hamiltonian diffeomorphisms is equal to the full group of symplectic diffeomorphisms. On the other hand, using $J$-holomorphic techniques, it was proved  in \cite{LP} that the stabilizer of this action, i.e. the subgroup of symplectic diffeomorphisms of  $M^i_\mu$ that preserve (not necessarily pointwise) a symplectically embedded ball  $B_c \subset \Mui$, can be identified, up to homotopy, with the group of all symplectomorphisms of $\tMuci$, the blow-up of $\Mui$ at a ball of capacity $c$ in $\Mui$. We therefore have the following fibration
\begin{equation}\label{FibrationPrincipale}
\Symp(\widetilde M^i_{\mu, c}) \hookrightarrow\Symp(M^i_{\mu})
\longrightarrow\IEmbuci
\end{equation}
that naturally expresses $\IEmb^i_{\om}(c,\mu)$ as an homogeneous space, namely 
\[
\IEmb^i_{\om}(c,\mu)\simeq\Symp(M^i_{\mu})/\Symp(\widetilde M^i_{\mu, c}). 
\]
Consequently, the homotopy-theoretic study of $\IEmbuci$ breaks down into three steps:
    
\begin{trivlist}
\item[\textit{Step 1.}] The computation of the homotopy type and cohomology algebra of $\Symp(M^i_{\mu})$ (as well as those of $B\Symp(M^i_{\mu})$). This step was carried out by a number of authors: Abreu \cite{Ab}, Abreu-McDuff \cite{AM}, Anjos \cite{An}, Anjos-Granja \cite{AG}, and Abreu-Granja-Kitchloo \cite{AGK}.

\vspace{5pt}
\item[\textit{Step 2.}] The computation of the homotopy type and cohomology algebra of $\Symp(\tMuci)$ (as well as those of $B\Symp(\tMuci)$). The rational cohomology modules $H^{*}(\Symp(\tMuci);\Q)$ and $H^{*}(\BSymp(\tMuci);\Q)$ were computed in \cite{LP} and \cite{Pi}. In the present paper, we will carry these calculations further and describe the full homotopy type of these spaces as well as the rational cohomology ring structure by adapting the arguments of~\cite{AGK}.

\vspace{5pt}
\item[\textit{Step 3.}] The most interesting step is understanding how $\Symp(\widetilde M^i_{\mu, c})$ \emph{sits inside} $\Symp(\Mui)$ so that we could compute the quotient. This step has been carried out in some special cases in~\cite{AL,LP,Pi}. In this article, we take a systematic approach to compute the rational homotopy type of the quotient. See Theorems~B.8  and B.9 in this paper showing that, even with the most natural choice of generators, the way in which $\Symp(\widetilde M^i_{\mu, c})$ sits inside $\Symp(\Mui)$ is not straighforward.

\end{trivlist}
\vspace{5pt}

Note that, in view of the fibration~(\ref{FibrationEmb-IEmb}) above, and since the reparametrization group of the standard ball $B^{4}\subset\R^{4}$ retracts to $U(2)$, the  computations for $\Im \Embuci$ carry easily to $\Embuci$. We get, in this way, similar theorems for the parametrized space of embeddings.

\subsection{The duality between $\Emb(\Muo)$ and $\Emb(\Mul)$} 
We now explain the duality introduced in \cite{Pi} that enables us to reduce the twisted case to the split one. 

   Denote by $B^0$ and $F^0$ in $H_2(M^0, \Z)$  the classes of the first and second factor respectively. Denote by $F^1$ the fiber of the fibration $M^1 (= \PbP)  \to \CP^1$ and by $B^1$ the section of self-interection $-1$ of that fibration. Denote by $E^i\in H_2( {\widetilde M}^i, \Z), i=0,1,$ the class of the exceptional divisor that one gets by blowing up the standard symplectic ball of capacity $c$ in $M^i$.  Let's denote by the same symbols $B^i, F^i$ the obvious lifts (proper transforms) of these classes to the blow-up ${\widetilde M}^i$.   Now let's recall the duality\footnote{That duality also exists on ruled symplectic 4-manifolds over surfaces of any genus and was exploited in \cite{La} to prove that the Non-Squeezing Theorem does not hold when the base of the trivial symplectic  fibration $\Sigma_g \times S^2$ is a real surface of genus greater than $0$.}  according to which  blowing up $M^0=S^2 \times S^2$ or $M^1=\PbP$ leads to diffeomorphic smooth manifolds ${\widetilde M}^0 \simeq{\widetilde M}^1 $. As explained in \cite{Pi}, the blow--down of an exceptional curve in ${\widetilde M}^0$ in class $F^0-E^0$ yields a manifold diffeomorphic to $\PbP$. The induced diffeomorphism between ${\widetilde M}^0$ and  ${\widetilde M}^1$  relates the basis $\{ B^0, F^0, E^0\}$ and  $\{ B^1, F^1, E^1\}$ as follows:
$$\begin{array}{ccc}
B^1 &\longleftrightarrow & B^0 -E^0 \\
E^1 &\longleftrightarrow & F^0 -E^0 \\
F^1 &\longleftrightarrow & F^0  
\end{array}$$
When one considers this birational equivalence in the symplectic category, the uniqueness of symplectic blow-ups implies that the blow-up of $M_\mu^0$ at a ball of capacity $0<c< \mu$ is symplectomorphic to the blow of  $M_{\mu-c}^1$ at a ball of capacity $1-c$. Conversely, the blow-up of  $M_\mu^1$ with capacity $0<c<1$ is symplectomorphic to the blow-up of 
$M_{\mu+1-c}^0$ with capacity $1-c$. In other words, we have a complete symplectic duality between the blow-up of ``large'' balls in $M^i$ and the blow-up of ``small'' balls in $M^{1-i}$. For this reason, we will state our results for both ruled surfaces $M^0$ and $M^1$ but we will often give the complete proof for the split case $M^0$ only, leaving to the reader its relatively easy adaptation (using the above equivalence) to the twisted case $M^1$.

\subsection{Plan of the paper} Here is an overview of the content of the paper. In Section~\ref{se:decompostion}, we briefly recall the geometric facts that lead to the homotopy decomposition of the groups of symplectomorphisms. The actual computations for the groups $\Symp(\tMuci)$ are carried in the Appendices, following the method introduced in Abreu-Granja-Kitchloo~\cite{AGK}. In Section~\ref{se:fibration}, we express rationally the space $\IEmbuci$ as a fibration whose base and fiber are computed. In Section~\ref{se:minimal-models} we compute the minimal model  of $\IEmbuci$, showing in particular that the latter space does not retract to a finite CW-complex for $\mu >1$ and $c \geq \ccrit$.  Finally, in Section~\ref{se:cohomology}, we compute the cohomology of $\IEmbuci$ with rational coefficients.

{\bf Acknowledgements.}
The authors would like to thank Gustavo Granja for useful conversations, Octav Cornea for discussions on some aspects of the theory of minimal models, and  V. Apostolov and A. Broer for conversations on complex algebraic geometry. But above all, the authors are grateful to the referee for reading the paper carefully and giving very pertinent suggestions, in particular for giving a way to correct the computation of the differential of $h$ in the minimal model of $ \Im \Emb_\omega(c,\mu)$.

\section{Homotopy decomposition of the symplectic groups} \label{se:decompostion}

This section is devoted to the homotopy decomposition of the groups $ \Symp(M^i_{\mu})$ and $\Symp(\widetilde M^i_{\mu, c})$. For the convenience of the reader, we first briefly review the geometric arguments that lead to the description of these symplectomorphism groups (and of their classifying spaces) as iterated homotopy pushouts. The references for this are the papers \cite{AGK, AM, LP, MD-Haefliger, MD-Stratification, Pi} and the two Appendices of the present paper in which we carry out the computations for the groups $\Symp(\tMuci)$.

To simplify the notations, we will write $\Gui$ and $\tGuci$ for the group $\Symp(M_\mu^i)$ and $\Symp(\tMuci)$.

\subsection{The limits $\lim_{\mu\to\infty}\Gui$ and $\lim_{\mu\to\infty}\tGuci$}\label{SubSectionInflation}
Let us first recall that the homotopy-theoretic understanding of the groups $\Gui$ and $\tGuci$ heavily relies on the generalization, due to McDuff, of the Lalonde-McDuff inflation technique. These ideas are used in McDuff~\cite{MD-Haefliger} to prove the following fundamental results. In the following two theorems, $\ell$ is the largest integer strictly smaller than $\mu$, i.e $\ell < \mu \le \ell + 1$.

\begin{thm}[See~\cite{MD-Haefliger, AGK}]\label{Inflation}
For any $\mu\geq 1$ and $\epsilon, \delta >0$, there is a natural diagram
\[
\xymatrix{
G_{\mu}^i \ar[r] \ar[dr] & G_{\mu+\epsilon}^i \ar[d] \\
 & G_{\mu+\epsilon+\delta}^i}
\]
well defined up to homotopy. Altogether, these maps define a homotopy coherent system whose homotopy limit is $\FDiff(M^i)$, the group of fibered $C^{\infty}$-diffeomorphisms, that is, those diffeomorphisms that are lifts to $M^i$ of diffeomorphisms of the base $S^2$ of the fibration $S^2 \hookrightarrow M^i \to S^2$. Moreover, 
\begin{enumerate}
\item the homotopy type of $\Gui$ remains constant as $\mu$ varies in the interval $(\ell,\ell+1]$.
\item The map $\Gui\to G_{\mu+\epsilon}^{i}$ is $(4\ell+2i-1)$-connected. In particular, when $\mu>1$, it induces an isomorphism of fundamental groups.
\item These maps induce surjections $H^{*}(BG_{\mu+\epsilon}^{i})\to H^{*}(BG_{\mu}^{i})$ for all coefficients. Consequently, the map $B\Gui\to\BFDiff$ induces a surjection in cohomology.
\end{enumerate}
\end{thm}
The same arguments can be adapted to the case of $\tGuci$ and yield

\begin{thm}[\cite{Pi}]\label{InflationBlowUp}
Given $c\in(0,1)$, there is a homotopy coherent system of maps 
\[
\xymatrix{
\tGuci \ar[r] \ar[dr] & \tG_{\mu+\epsilon,c}^{i} \ar[d] \\
 & \tG_{\mu+\epsilon+\delta,c}^{i}}
\]
defined for all $\mu\geq 1$ and all $\epsilon, \delta >0$, whose homotopy limit is $\FDiff_{*}(M^i)$, the group of fibered $C^{\infty}$-diffeomorphisms of $M^{i}$ that fix a point $p\in M^{i}$. Moreover, 
\begin{enumerate}
\item the homotopy type of $\tGuci$ remains constant as $\mu$ varies in either  $(\ell,\ell + c)$ or $[\ell + c, \ell+1]$.
\item The map $\tGuci\to\tG_{\mu+\epsilon,c}^{i}$ is $(4\ell+2i-3)$-connected if $c\geq\ccrit$, and $(4\ell+2i-1)$-connected if $c<\ccrit$. In particular,  when $\mu>1$, it induces an isomorphism of fundamental groups.
\item These maps induce surjections $H^{*}(B\tG_{\mu+\epsilon,c}^{i})\to H^{*}(B\tG_{\mu,c}^{i})$ for all coefficients. Consequently, the map $B\tGuci\to\BFDiff_{*}$ induces a surjection in cohomology.
\end{enumerate}
\end{thm}
\begin{proof}
See Theorem~1.3 and Proposition~3.6 in~\cite{Pi}.
\end{proof}

\subsection{The action of $\Symp(\Mui)$ on compatible almost complex structures} 

Let $\jjui$ be the space of all $C^{\infty}$-almost complex structures compatible with the symplectic form $\omui$ on $\Mui$. This is an infinite dimensional Fr\'echet manifold on which $\Gui$ acts by conjugation, that is,  $ \phi \cdot J \equiv (d\phi) J (d\phi)^{-1}$  where $\phi \in\Gui$  and $ J \in \jjui$. Observe that because $\jjui$ is contractible, the associated homotopy orbit space (i.e. the Borel construction)
\[
\big(\jjui\big)_{h\Gui}:=E\Gui \times_{\Gui} \jjui
\]
is homotopy equivalent to the classifying space $\BGui$. It is a standard fact that the projections yield an equivariant diagram
\[
\xymatrix{
E\Gui \times \jjui   \ar[r] \ar[d] & \jjui \ar[d] \\
\BGui \ar[r]^-{\phi} &  \jjui/\Gui
}
\]
such that the preimage $\phi^{-1}[J]$ is naturally identified with the classiying space $BK_{J}$ of the stabilizer subgroup $K_{J}$ of $J$. In our case, this isotropy subgroup is the group of isometries of the almost Hermitian structure associated to the pair $(\omui,J)$ and, hence, is a compact Lie group. Moreover, as we will explain below, the orbit category associated to the action of $\Gui$ on $\jjui$ is essentially finite and can be understood by combining $J$-holomorphic techniques with standard results from the theory of deformation of complex structures. This leads to a description of $\BGui$ in terms of  classifying spaces $\BK_{J}$ of finitely many compact Lie subgroups $K_{J}\subset\Gui$. 
 
\subsection{The stratification of $\jjui$ as an orbit decomposition}

The space $\jjui$ is naturally partitioned in $(\ell+1)$ strata indexed by even integers in the split case $i=0$ and by odd integers in the twisted case $i=1$:
\[
\jj_{\mu}^{i}=
\jj_{\mu,i}^{i}\sqcup\jj_{\mu,2+i}^{i}\sqcup\cdots\sqcup \jj_{\mu,2\ell+i}^{i}
\]
where as usual $\ell$ is the largest integer strictly smaller than $\mu$. The stratum $\jj_{\mu,2k+i}^{i}$ is made of those almost complex structures $J$ for which the class $B^i - k F^i$ can be represented by an embedded $J$-holomorphic $2$-sphere. Note that this is indeed a partition: by positivity of intersection, a $J$-structure cannot belong to more than one such stratum, and any $J \in \jjui$ must belong to at least one stratum since the GW-invariant associated to the class $B^{i}$ does not vanish (use then the Gromov compactness theorem to conclude). Each stratum is a smooth co-oriented Fr\'echet submanifold of finite codimension: the stratum $\jj_{\mu,i}^{i}$ is an open and dense subset of $\jjui$ while for $j=2k+i\geq 2$ the stratum $\jj_{\mu,j}^{i}$ is of codimension $2j-2$. The reader will find in \cite{AM} the proofs of the results regarding the stratification of $\jjui$ and further references.

Each stratum corresponds to a toric structure on $\Mui$, unique up to equivariant symplectomorphisms. In particular, $\jj_{\mu,j}^{i}$ contains a Hirzebruch complex structure $J_{j}$, unique up to diffeomorphisms, coming from an identification of $(\Mui,J_{j})$ with the Hirzebruch surface $\F_{j}:=\mathbb{P}(\Oo(-j)\oplus\C)$ (hence our choice of indices). The stabilizer subgroup $K(j)$ of $J_{j}$ is given, up to isomorphism, by: 
\[
K(j)\simeq
\begin{cases}
\SO(3)\times\SO(3) & \text{if $j=0$,}\\
S^{1}\times\SO(3) & \text{if $j=2k$, $k\geq 1$,}\\
\U(2) & \text{if $j=2k+1$, $k\geq 0$.}
\end{cases}
\]
The closure of $\jj_{\mu,j}^{i}$ in $\jjui$ is the union of all strata of index $n\geq j$ 
\[
\overline{\jj}_{\mu,j}^{i}:=\jj_{\mu,j}^{i}\sqcup\cdots\sqcup\jj_{\mu,2\ell+i}^{i}\]
In fact, using $J$-holomorphic gluing techniques, one can show that the partition is a genuine stratification: each $\jj_{\mu,j}$ has a neighborhood $\Nn_{j} \subset\jj_{\mu}$ which, once given the induced stratification, has the structure of a locally trivial fiber bundle whose typical fiber is a cone over a finite dimensional stratified space.

Most importantly, the action of $\Gui$ preserves each stratum and, although the action restricted to a stratum cannot be transitive (because, for instance, each stratum contains both integrable and non-integrable structures), the inclusion 
\[
\Gui/K(j)\into \jj_{\mu,j}^{i}
\]
of the symplectic orbit of $J_{j}$ in $\jj_{\mu,j}^{i}$ is a weak homotopy equivalence. 

Let us consider the particular case $\Muo=(S^{2}\times S^{2},\mu\sigma\oplus\sigma)$ with 
$1<\mu\leq 2$ more closely. For $\mu$ in that range, 
the space $\jjuo$ is made of an open stratum $\jj_{\mu,0}^{0}\simeq G_{\mu}^{0}/K_{0}$ and a codimension $2$ stratum $\jj_{\mu,2}^{0}\simeq G_{\mu}^{0}/K(2)$. The isotropy groups intersect along a common $\SO(3)$ which is the $\SO(3)$ factor in $K(2)=S^{1}\times\SO(3)$ and the diagonal $\SO(3)$ in $K(0)=\SO(3)\times\SO(3)$. The action of the $S^{1}$ factor of $K(2)$ on a fiber of the normal bundle of $\jj_{\mu,2}^{0}$ is isomorphic to the standard linear action of $S^{1}$ on $\R^{2}$. In particular, $K(2)$ acts transitively on the unit normal bundle over $J_{2}$ with stabilizer $\SO(3)$. Now assume that there exists a $G_{\mu}^{0}$-invariant tubular neighborhood $\Nn:=\Nn(\jj_{\mu,2}^{0})$ of $\jj_{\mu,2}^{0}$ in $\jj_{\mu}^{0}$ isomorphic to a tube $G_{\mu}^{0}\times_{K(2)}D^{2}$. Then we could write the contractible space $\jj_{\mu}^{0}$ as an equivariant homotopy pushout
\begin{equation}
\label{PushoutDiagramSimpleCase-1}
\xymatrix{
\Nn(\jj_{\mu,2}^{0}) - \jj_{\mu,2}^{0}  \ar[r] \ar[d] & \jj_{\mu,2}^{0} \simeq G_{\mu}^{0}/ K(2) \ar[d] \\
\jj^{0}_{\mu,0}\simeq G_{\mu}^{0}/K(0) \ar[r] &  \jj_\mu^{0}
}
\end{equation}
and, by applying the Borel construction $EG_{\mu}^{0} \times_{G_{\mu}^{0}}$, we would get another pushout diagram
\begin{equation}
\label{PushoutDiagramSimpleCase-2}
\xymatrix{
B\SO(3)  \ar[r] \ar[d] & B(S^{1}\times\SO(3)) \ar[d] \\
B(\SO(3)\times\SO(3)) \ar[r] &  BG_{\mu}^{0}   
}
\end{equation}
that would decompose (the homotopy type of) $B\Guo$ along conjugacy classes of maximal compact subgroups. The only problem with this argument is that it may be impossible to construct such an invariant tubular neighborhood $\Nn$. Nevertheless, as explained in~\cite{AGK} Appendix D, a slice theorem for the action of $\Gui$ on $\jjui$ allows one to make the previous argument completely rigorous\,\footnote{See also~\cite{AG} for a different, more algebraic, approach.} by defining, for any indices $i$ and $j$, an $A_{\infty}$-action of $\Gui$ near $\jj_{\mu,j}^{i}$ which is essentially equivalent to the left action of $\Gui$ on a tube $\Gui\times_{K(j)}D^{2j-2}$. 

In the general case $\mu>1$, $i\in\{0,1\}$, one may decompose $\jjui$ as the union
\[
\big(\jj_{\mu,i}^{i}\sqcup\cdots\sqcup\jj_{\mu,2\ell+i-2}^{i}\big)\cup\Nn(\jj_{2\ell+i}^{i})
\]
To apply the previous ideas to this decomposition, one has to understand the action of $K(2\ell+i)$ on the normal bundle $\Nn(\jj_{2\ell+i}^{i})$ of the last stratum and one must compute the homotopy orbit space 
\[\big(\jj_{\mu,i}^{i}\sqcup\cdots\sqcup\jj_{\mu,2\ell+i-2}^{i}\big)_{h\Gui}.\] 
In principle, this can be done using $J$-holomorphic gluing techniques but, as explained in~\cite{MD-Stratification}, the computations quickly become intractable as $\mu$ increases. A solution to this problem, found by Abreu-Granja-Kitchloo in~\cite{AGK}, is to look at the restriction of the action $\Gui\times\jjui\to\jjui$ to the subspace $\Ii_{\mu}^{i}\subset\jjui$ of compatible \emph{integrable} complex structures. As they explained, the point is that for K\"ahler $4$-manifolds satisfying some analytical conditions, the action of the symplectomorphism group on the space of compatible integrable complex structures can be understood using complex deformation theory. In the special case of rational ruled surfaces $\Mui$, they showed that
\begin{enumerate}
\item $\Ii_{\mu}^{i}$ is a submanifold of $\jjui$ and the inclusion $\Ii_{\mu}^{i}\subset\jjui$ is transverse to each stratum $\jj_{\mu,j}^{i}$.
\item The stratum $ \Ii_{\mu,j}^{i}:=\Ii_{\mu}^{i}\cap\jj_{\mu,j}^{i}$ is homotopy equivalent to the symplectic orbit of $J_{j}$ in $\Ii_{\mu,j}^{i}$.
\item For any $J\in\Ii_{\mu}^{i}$,
the tangent space of $\Ii_{\om}^i$ at $J$ is naturally identified with   
$T_{J}((\Diff(M) \cdot J)\cap \Ii_{\om}^i)\oplus H^{0,1}_{J}(T\Mui)$, where $T\Mui$ denotes the sheaf of germs of holomorphic vector fields. Here the (infinite dimensional) first summand is the tangent space to the stratum $\Ii_{\mu,j}^{i}$ to which $J$ belongs, while the (finite dimensional) second summand is naturally identified with the  fiber of the normal bundle of $\Ii_{\mu,j}^{i}\subset \Ii_{\mu}^{i}$ at $J$.
 Moreover, this identification is equivariant with respect to the action of the isometry group $K(j)$.
\end{enumerate}
It follows that $\Ii_{\mu}^{i}$ is itself a stratified space, that the inclusion $\Ii_{\mu,j}^{i}\into\jj_{\mu,j}^{i}$ is a homotopy equivalence, that the equivariant diffeomorphism type of a normal neighborhood of the $j^{\text{th}}$ stratum is the same in both stratifications, and that this neighborhood does not depend on the parameter~$\mu$ as long as $\mu> (j-i)/2$. These facts, together with the results of Appendix D in~\cite{AGK}, imply that the action of $\Gui$ on the normal bundle $\Nn(\jj_{2\ell+i}^{i})$ is homotopically equivalent to the left action of $\Gui$ on the tube $\Gui\times_{K(2\ell+i)} H^{0,1}_{J_{2\ell+i}}(T\Mui)$, and that the homotopy orbit space
\[\big(\jj_{\mu,i}^{i}\sqcup\cdots\sqcup\jj_{\mu,2\ell+i-2}^{i}\big)_{h\Gui}\] 
can be understood iteratively. Finally, because $\dim_{\C} H^{0,1}_{J_{j}}(T\Mui) = 2j-2$, we get the following homotopy decomposition of $B\Gui$:
\begin{thm}[\cite{AGK}, Theorem 5.5]\label{thm:HomotopyPushout}  There is a homotopy pushout diagram
$$
\xymatrix{
\big( S^{4\ell+2i-3}\big)_{hK(2\ell+i)} \ar[d]^{j_{\ell}} \ar[r]^-{\pi} & ~~BK(2\ell+i)~~  \ar[d]^{i_{\ell}} \\
BG_{\mu-1}^{i} \ar[r]^{} & BG_{\mu}^{i}  
}
$$
where $ \ell < \mu \leq \ell +1$,  $\pi$ is the bundle associated to the representation of $K(2\ell+i)$ on $H^{0,1}_{J_{2\ell+i}}(T\Mui)$,  $i_{\ell}$ is induced by the inclusion $K(2\ell+i)\into \Gui$, and where the map $BG_{\mu -1}^{i}\to BG_{\mu}^{i}$ coincides, up to homotopy, with the one given by inflation  described in Theorem~\ref{Inflation}.
\end{thm}

\subsection{Homotopy decomposition of $B\tGuci$} 

There is  a similar pushout decomposition for the classifying space of $\tGuci$. 
Note that because we can identify symplectically $\tMuco$ with $\widetilde M^1_{\mu-c,1-c}$, there is no loss in generality to restrict ourselves to the case $i=0$, that is, to symplectic blow-ups of the trivial bundle $\Muo$ only. 

All of Abreu-Granja-Kitchloo arguments on $M^0_{\mu}$ apply as well for the group of symplectomorphisms of the blow-up $\tMuco$ if one has in mind the following geometric facts and observations. When passing to the blow-up, the spaces of compatible (almost) complex structures $\tjjuco$ and $\tIi_{\mu,c}^{0}$ are now partitioned according to the degeneracy type of exceptional curves in class $B-E$ (using the notation introduced in Section~1.3). Indeed, recall that there are exactly three exceptional classes in $H_{2}(\tMuco)$, namely $E$, $F-E$, and $B-E$. For generic $J$, they are all represented by embedded $J$-holomorphic spheres. However, when $\mu>1$, the class $B-E$ has strictly larger area than $E$ and $F-E$ and it follows that 
\begin{enumerate}
\item the exceptional classes $E$ and $F-E$ are symplectically indecomposable and, given any $J\in\tjjuco$ , are always represented by embedded $J$-holomorphic spheres. 
\item A $J$-holomorphic representative of the exceptional class $D_0 := B-E$ can only degenerate, as $J$ varies, to a cusp-curve containing a unique embedded representative of either $D_{2k-1}:=B-kF$ or $D_{2k}:=B-kF-E$ for some $1\leq k\leq \ell$. 
\end{enumerate}
Because the intersection $D_{i}\cdot D_{j}$ is always negative, this defines a partition of spaces of compatible (almost) complex structures in which the $j$\textsuperscript{th} stratum $\widetilde{\jj}_{\mu,c,j}^{0}$ consists of those complex structures $J$ for which the class $D_{j}$ admits an embedded $J$-holomorphic representative. The set of strata is in bijection with the set of equivalence classes of toric actions on $\tMuco$ and incorporates strata coming, after blow-up, from both $\jjuo$ and $\jj_{\mu-c}^{1}$. It is easy to see that the strata are now indexed by all integers between $0$ and $m$, where $m=2\ell$ if $c<\ccrit$ or $m=2\ell-1$ if $c\ge \ccrit$. In particular, when $c$ belongs to the range $[\ccrit, 1)$, which is the case considered in this paper, this stratification starts at the dense stratum associated to $D_0$ and ends at the $(2\ell-1)$\textsuperscript{th} stratum associated to the class $D_{2\ell - 1} = B-\ell F$. Again, the symplectomorphism group $\tGuco$ acts on $\tjjuco$ preserving the stratification. 
\begin{prop}[see \cite{LP} \S4, and~\cite{MD-Stratification}]\label{StructureStratificationBlowUp} Given $\mu\geq1$ and $c\in(0,1)$, recall that $\ell$ is such that $\ell < \mu \le \ell+1$. Let $m$ be the index of the last stratum in $\tjjuco$, namely
\[
m:=
\begin{cases}
2\ell   &\text{~if $c<\ccrit$}\\
2\ell-1 &\text{~if $c \geq \ccrit$}
\end{cases}
\]
Then, given $1\leq j\leq m$, we have
\begin{enumerate}
\item The subspace $\widetilde{\jj}_{\mu,c,j}^{0}$ is a smooth, co-oriented, codimension $2j$  submanifold whose closure is the union $\bigsqcup_{j\leq s\leq m}\widetilde{\jj}_{\mu,c,s}$.
\item The stratum $\widetilde{\jj}_{\mu,c,j}$ contains a complex structure $\widetilde{J}_{j}$, unique up to diffeomorphisms, coming from the blow-up of the Hirzebruch surface $\F_{j}$ at a point $p$ belonging to the zero section. Equivalently, one can obtain $\widetilde{J}_{j}$ by blowing up the even Hirzebruch surface $\F_{2k}$ on the zero section if $j=2k$, or on the section at infinity if $j=2k-1$.
\item The group $\tGuco$ acts smoothly on $\widetilde{\jj}_{\mu,c,j}^{0}$. The stabilizer of $\widetilde{J}_{j}$ is the $2$-torus $\widetilde{T}(j)$ generated by the lifts of the K\"ahlerian isometries of $\F_{j}$ fixing the center $p$ of the blow-up. This identifies $\widetilde{T}(j)$ with a maximal torus of $K(j)$.
\item The inclusion of the symplectic orbit $\tGuco\cdot\widetilde{J}_{j}\simeq\tGuco/\widetilde{T}(j)\into\widetilde{\jj}_{\mu,c,j}^{0}$ is a homotopy equivalence.
\end{enumerate}
\end{prop}
As we explain in Appendix A, the action of the symplectomorphism group on $\tjjuco$ is homotopically equivalent to its restriction to the subset of compatible integrable complex structures $\tIi_{\mu,c}^{0}$. Since the last stratum has real codimension $2m$, this yields the following description of $B\tGuco$:
\begin{thm}\label{thm:HomotopyPushoutBlowUp}
If $\ell < \mu \leq \ell+1 $ and $c \in (0,1)$, there is a homotopy pushout diagram
$$
\xymatrix{
S^{2m-1}_{h\widetilde{T}(m)} \ar[r] \ar[d]^{j_{m}} & B\widetilde{T}(m) \ar[d]^{i_{m}} \\
B{\widetilde G}_{\mu',c}^0 \ar[r] & B\tGuc^0}
$$
where $m$ is the index of the last stratum of $\tjjuco$, and where
\[
\mu'=
\begin{cases}
\ell +c & \text{if $c<\ccrit$}\\
\ell   & \text{if $c \geq \ccrit$}
\end{cases}
\]
so that $\tG_{\mu',c}^{0}$ is the group associated with a stratification having one stratum less than the stratification associated with $\tG_{\mu,c}^{0}$. The upper horizontal map  is the universal bundle map associated to the representation of $\widetilde{T}(m)$ on $H^{0,1}_{J_{m}}(T\tMuci)$, $i_{m}$ is induced by the inclusion $\widetilde{T}(m)\into \tGuc^0$, and the map $B\tG_{\mu',c}^{0}\to B\tGuc^0$ coincides, up to homotopy, with the one given by inflation  described in Theorem~\ref{InflationBlowUp}.
\end{thm}

\section{Homotopy type of the space of embedded symplectic balls}\label{se:fibration}

In this section we describe the rational homotopy type of the space
\[
\IEmb(B_{c},\Mui)\simeq \Gui/\tGuci
\]
as the total space of a fibration whose base and fiber are explicitely computed. We will prove the following theorem:

\begin{thm}\label{thm:homotopy-type} 
If $c \ge \ccrit$, the topological space $\IEmb_{\om}^{i}(c,\mu)$ has the rational  homotopy type of the total space of a fibration 
$$ F^i \to  G_\mu^i /\tGuc^i \to M_{\mu}^i,$$
where $F^0= S^{4\ell-1} \times \Omega S^{4\ell+1}$ and $F^1=S^{4\ell+1} \times \Omega S^{4\ell+3}$ as topological spaces. The projection map is homotopy equivalent to the push-forward, through the quotient map
$\Emb_{\om}^{i}(c,\mu) \to \IEmb_{\om}^{i}(c,\mu)$, of the evaluation map at the center of the ball $ev_{center}: \Emb_{\om}^{i}(c,\mu) \to  M_{\mu}^i$. \end{thm}

To prove the theorem~\ref{thm:homotopy-type}, it is convenient to consider the untwisted case and the twisted case separately.

\subsection{The untwisted case}\label{subsection:untwisted} 
Let $\FDiff$ be the group of fiber preserving diffeomorphisms of $S^{2}\times S^{2}$, and $\FDiff_{*}\subset\FDiff$ be the stabilizer of a point. Since $\FDiff$ acts transitively on $S^{2}\times S^{2}$, there is a fibration
\[
S^{2}\times S^{2}\to\BFDiff_{*}\to\BFDiff 
\]
and because $\FDiff\simeq\hocolim_{\mu\to\infty}\Guo$ and $\FDiff_{*}\simeq\hocolim_{\mu\to\infty}\tGuco$, there is a homotopy commuting diagram of fibrations

\begin{equation}\label{ExtendedSquare}
\xymatrix{
F_{\eta}  \ar[r] \ar[d] &  G_{\mu}^0/\tG_{\mu,c}^0  \ar[r]^{\eta} \ar[d] &  S^{2}\times S^{2} \ar[d] \\ 
\widetilde{F}_{\tpsi_{\mu,c}} \ar[r] \ar[d]  & \tBG_{\mu,c}^0 \ar[d]  \ar[r]^{\tpsi_{\mu,c}} & \BFDiff_{*} \ar[d] \\ 
F_{\psi_{\mu}} \ar[r]  & B\Guo \ar[r]^{\psi_{\mu}} & \BFDiff
}
\end{equation}
in which the spaces in the leftmost column are defined as the homotopy fibers of the horizontal maps. 
Over the rationals, this diagram simplifies enough to allow explicit computations. For instance, the topological group $\FDiff$ is homotopy equivalent to the semi-direct product $\SO(3)\ltimes\Map(S^2,\SO(3))$ where $\SO(3)$ acts on $\Map(S^{2},\SO(3))$ by precomposition. In fact, the principal fibrations 
\[\Map(S^{2},\SO(3))\to\FDiff\to\SO(3)\]
and
\[\Omega^{2}\SO(3)\to\Map(S^{2},\SO(3))\to\SO(3)\]
both admit sections so that, as a space, $\FDiff\simeq\Omega^{2}\SO(3)\times\SO(3)\times\SO(3)$.
At the classifying space level, we have fibrations with natural sections
\[\Map(S^{2},\BSO(3))\to\BFDiff\to\BSO(3)\]
and
\[\Omega\SO(3)\to\Map(S^{2},\BSO(3))\to\BSO(3).\]
Because the rational cohomologies of $\BSO(3)$ and $\Omega\SO(3)$ are concentrated in even degrees, the corresponding rational spectral sequences collapse at the second stage, and since
\[
H^{*}(\Omega\SO(3)\times\BSO(3)\times\BSO(3);\Q) =H^{*}(K(\Q,2)\times K(\Q,4)\times K(\Q,4)),
\] 
it follows that there are rational homotopy equivalences
\[
\BFDiff\to K(\Q,2)\times K(\Q,4)\times K(\Q,4)\from\Omega\SO(3)\times\BSO(3)\times\BSO(3)
\] 
In fact, since $\Omega\SO(3)\simeq_{\Q}BS^{1}$, there is a natural map
\[
BS^{1}\vee\BSO(3)\vee\BSO(3)\to\BFDiff
\]
that, rationally, extends to a homotopy equivalence
\[
BS^{1}\times\BSO(3)\times\BSO(3)\to\BFDiff.
\]
Note that the same arguments as above show that the classifying space of the stabilizer subgroup $\FDiff_{*}$ is rationally equivalent to $BS^{1}\times BS^{1}\times BS^{1}\simeq K(\Q,2)\times K(\Q,2)\times K(\Q,2)$.

\begin{lemma} Over the rationals, $BG_{\mu}^{0}$ fibers over $B(\SO(3)\times\SO(3) \times S^1)$ with fiber $S^{4\ell+1}$.
\end{lemma}
\begin{proof} 
We know from Abreu-Granja-Kitchloo~\cite{AGK} that the rational cohomology ring of $BG_{\mu}^{0}$ is isomorphic to 
$$\Q[T,X,Y]/\langle f\rangle $$
where the generators are of even degrees $|T|=2$, $|X|=4$, and $|Y|=4$, and where $f$ is an homogeneous polynomial of degree $4\ell+2$. The theory of minimal models (see, for instance, the discussions in the begining of sections \S4 and \S5) implies that, rationaly, the cohomology ring of the homotopy fiber of the map
$$ BG_{\mu}^{0}\to K(\Q,2)\times K(\Q,4)\times K(\Q,4)$$
is isomorphic to an exterior algebra with a single generator of degree $4\ell+1$. Therefore, the homotopy fiber is rationaly equivalent to $K(\Q,4\ell+1)\simeq S^{4\ell+1}$. 
\end{proof}

Similarly, the description of the rational cohomology ring of $B\tG_{\mu,c}^0$ given by theorem~\ref{thm:RationalCohomologyAlgebra} in the Appendix yields
\begin{lemma} Over the rationals, the space $B\tG_{\mu,c}^0$ fibers over $B(S^1 \times S^1 \times S^1)$ with fiber $S^{4\ell - 1}$.
\end{lemma}
The previous two lemmas implies that the diagram~(\ref{ExtendedSquare}) is homotopy equivalent, over the rationals, to the following commutative diagram in which $\IEmb_\omega^0(c,\mu)\simeq G_\mu^0 /\tGuc^0 $ appears, as desired, as the total space of a fibration whose base and fiber are known: 
\begin{equation}\label{FullDiagramRational}
\xymatrix{
F^0\ar[r] \ar[d] & G_\mu^0 /\tGuc^0 \ar[r] \ar[d] & S^2 \times S^2 \ar[d] \\
S^{4\ell-1}\ar[r] \ar[d]^j & B\tGuc^0 \ar[r] \ar[d] & B(S^1 \times S^1 \times S^1)\ar[d] \\ 
 S^{4\ell+1}\ar[r]& BG_\mu^0  \ar[r] & B(\SO(3) \times \SO(3) \times S^1)}
\end{equation}
Notice that $F^0$ is the fiber of the map $j$ from $S^{4\ell-1}$ to $S^{4\ell+1}$. Any map between such spheres is null homotopic, so $F^0= S^{4\ell-1} \times \Omega S^{4\ell+1}$ as topological spaces. This proves theorem~\ref{thm:homotopy-type} in the untwisted case.

\subsubsection{The particular case $1 < \mu \leq 2$} \label{se:ParticularCase}
When $1 < \mu \leq 2$, one can strengthen theorem~\ref{thm:homotopy-type} by computing the full homotopy type of the embedding space $\Im\Emb_\omega^0(c,\mu)$. This range of $\mu$ corresponds to the first step of the induction process that gives the homotopy type of $BG_\mu^0 $ and $B\widetilde G_{\mu,c}^0 $ as pushout squares.  In this case the Borel construction $S^{4\ell-3}_{hK(2\ell)}=EK(2\ell) \times_{K(2\ell)}S^{4\ell-3}$ gives $S^1_{hK(2)} \simeq B\SO(3)$. Therefore, it is easy to see that there are maps $\psi_0$ and $\psi_1$ that make the following diagram commutative. 
$$
\xymatrix{
B\SO(3) \ar[d]^-{\Delta} \ar[r]^-{\pi} & ~~B(\SO(3) \times S^1) ~~  \ar[d]^{i_1} \ar[ddr]^{\psi_1} &\\
B(\SO(3) \times \SO(3))  \ar[r]^-{i_0} \ar[drr]^{\psi_0} & BG_\mu^0   \ar@{-->}[dr]  & \\
  &  & B(\SO(3) \times \SO(3)\times S^1)}
$$
where $\Delta$ is the diagonal map, $\pi$ is the inclusion of the first factor, $i_0$  and $i_1$ are the inclusions of the classifying spaces of the isotropy subgroups. Note that this diagram holds not only over the rationals 
but also  over the integers.

Similarly, if $0<\mu-1 \leq c <1$, then the homotopy orbit $S^1_{hT(1)}$ is equivalent to  $BS^1$ and we get the following commutative diagram that also holds over the integers. 
$$
\xymatrix{
BS^1 \ar[d]^{\Delta} \ar[r]^-{\pi} & ~~B(S^1 \times S^1) ~~  \ar[d]^{i_1} \ar[ddr]^{\widetilde{\psi}_1} &\\
B(S^1 \times S^1)  \ar[r]^-{i_0} \ar[drr]^{\widetilde{\psi}_0} & B\widetilde G_{\mu,c}^0   \ar@{-->}[dr]  & \\
  &  & B(S^1 \times S^1\times S^1)}
$$
The fibers of the maps $BG_\mu^0 \rightarrow B(\SO(3) \times \SO(3) \times S^1)$ and $B\widetilde{G}_{\mu,c}^0 \rightarrow B(S^1 \times S^1 \times S^1)$ are given by $\Sigma^2 \SO(3)$ and $S^3$, respectively. Using these fibrations, we can then construct a commutative diagram, as in \eqref{FullDiagramRational}, that now gives the full homotopy type of the space of embedded balls.

\begin{thm}\label{thm:homotopy-type-mu<2}
If $0<\mu-1 \leq c < 1$, the topological space $\Im\Emb_\omega^0(c,\mu)$ has the full homotopy type of the total space of a fibration
\begin{equation}
\label{eq:homotopy-type-mu<2}
 \Omega \Sigma^2 \SO(3)/ \Omega S^3 \to \Im\Emb_\omega^0(c,\mu) \to S^2\times S^2.
\end{equation} 
where the inclusion $\Omega S^3 \subset \Omega \Sigma^2 \SO(3)$ is understood by identifying $\Omega S^3$ with $\Omega \Sigma^2 S^1$ and taking the standard inclusion of $S^1$ in $SO(3)$.

 Moreover, $S^2\times S^2$ is a retract of the space of embedded balls.
\end{thm}
\proof{}
Since $S^2\times S^2$ may be identified with the homogeneous space $(\SO(3) \times \SO(3)) / (S^1 \times S^1)$ where these two groups are subgroups of $G_\mu^0$ and $\widetilde{G}_{\mu,c}^0$, respectively, the fibration has a section. This proves the second statement in the theorem. 
\qed

\subsection{The twisted case} There is a whole similar picture for the twisted bundle $\Mul=(\NTB,\om_{\mu})$. Let us write $\FDiff$ for the group of fiber preserving diffeomorphisms of $\Mul$, and $\FDiff_{*}\subset\FDiff$ for the stabilizer of a point. Rationally, we have homotopy equivalences
\begin{gather}
\BFDiff\simeq \BSU(2)\times\BSU(2)\times BS^{1}\\
\BFDiff_{*}\simeq BS^{1}\times BS^{1}\times BS^{1}
\end{gather}
Over the rationals, this yields a commutative diagram that expresses the homotopy type of $\IEmb_{\om}^{1}(c,\mu)\simeq\Gul/\tGucl$ as the total space of a fibration:
\[
\xymatrix{
S^{4\ell+1} \times \Omega S^{4\ell+3}\ar[r] \ar[d] & G_\mu^1 /\tGuc^1 \ar[r] \ar[d] & \NTB \ar[d] \\
S^{4\ell+1} \ar[r] \ar[d]^j & B\tGuc^1 \ar[r] \ar[d] & B(S^1 \times S^1 \times S^1)\ar[d] \\ 
S^{4\ell+3} \ar[r]& BG_\mu^1  \ar[r] & B(SU(2) \times SU(2) \times S^1)}
\]

   This concludes the proof of Theorem~\ref{thm:homotopy-type}.
   
\section{The minimal models of $\Symp(M_{\mu}^i)$ and of $\Symp(\tMuci)$}\label{sc:mm1}

First recall that in order to be applicable to some given topological space, the  theory of minimal models does not require that the space be simply connected. We simply need that the space  has a nilpotent homotopy system, which means that $\pi_1$ is nilpotent and $\pi_n$ is a nilpotent $\pi_1$--module for $n> 1$. Since the groups of symplectomorphisms $\Symp(M_{\mu}^i$) and  $\Symp(\tMuci)$ are $H$--spaces, it follows that they have a nilpotent homotopy system, because for a $H$--space, $\pi_1$ is abelian and is therefore nilpotent, and moreover  $\pi_1$ acts trivially on all $\pi_n$'s. On the other hand, $\Im \Emb^i$ is simply connected since we know that the generators of $\pi_1(\Symp(M_{\mu}^i))$ lift to the generators of $\pi_1(\Symp({\widetilde{M}}_{\mu,c}^i))$.  Therefore the theory of minimal models is applicable to all spaces under consideration.

   Recall that a model for a space $X$ is a graded differential algebra that provides a complete rational homotopy invariant of the space. Its cohomology is the rational cohomology of the space. The model can be constructed from the rational homotopy groups of $X$. In this case, it is always minimal, which implies that there is no linear term in the differential of the model, i.e the first term is quadratic. When there are no higher order term (i.e all terms are quadratic), then Sullivan's duality can be expressed in the following way:
$$
d b_k =  \sum_{i,j} \langle b_k,[b_i, b_j]\rangle    b_i b_j.
$$
where the $\langle a,b\rangle$ denotes the $a$-coefficient in the expression of $b$, and where  the brackets denote the Whitehead product.  Finally, when $X$ is an H-space, as it is the case of both $\Symp(M_{\mu}^i)$
 and $\Symp(\tMuci)$, all Whitehead products vanish as well as the differential.

   From these considerations and the computations of the rational homotopy groups of both $\Symp(M_{\mu}^i)$ and of $\Symp(\tMuci)$ in \cite{AM,Pi}, we have:

\begin{itemize}

\item The minimal model of  $\Symp({\widetilde{M}}_{\mu,c}^i)$  is $\Lambda ({\widetilde t},{\widetilde x},{\widetilde y}, {\widetilde w})$, the free graded algebra generated by the elements ${\widetilde t},{\widetilde x},{\widetilde y},{\widetilde w}$ with degrees $\deg {\widetilde t} = \deg {\widetilde x} = \deg {\widetilde y} = 1 $ and $ \deg {\widetilde w} =4\ell+2i-2$. 

\item The minimal model of $\Symp(M_{\mu}^i)$ is $\Lambda (t,x,y, w)$, the free graded algebra generated by the elements $t,x,y,w$ with degrees $\deg t = 1$, $\deg x = \deg y = 3$ and $ \deg w =4\ell+2i$.

\end{itemize}

 Let's now explain what these generators are. 

First recall that  $F$ (or more precisely $F^i$) denotes  the homology class of the fiber of $M^i_\mu$,  and  $E$ (or more precisely $E^i$) $\in H_2(\tMuci,\Z)$ is the class of the exceptional divisor that one gets by blowing up the standard symplectic ball of capacity $c$ in  $M^i_\mu$. 

\medskip
We first briefly recall the definition of the Hirzebruch surfaces.
    For any $\nu > 0$ and any integer $k \ge 0$ satisfying $\nu - \frac
k2 > 0$,
let $\CP^1 \times
\CP^2$ be endowed with the K\"ahler form $(\nu - \frac k2) \tau_1 +
\tau_2$ where
$\tau_{\ell}$ is the Fubini-Study form on $\CP^{\ell}$ normalized so
that the area of
the linear $\CP^1$'s be equal to $1$. Let
$\F_k$ be the corresponding Hirzebruch surface, i.e. the K\"ahler surface
defined by
$$
\F_k = \{([z_0,z_1],[w_0,w_1,w_2]) \in \CP^1 \times \CP^2 \ | \ z_0^k
w_1 = z_1^k w_0 \}
$$
It is well-known that the restriction of the projection $\pi_1:
(\CP^1 \times \CP^2, (\nu - \frac k2) \tau_1 +
\tau_2)  \to \CP^1$ to $\F_k$ endows  $\F_k$ with the structure of a
K\"ahler
$\CP^1$-bundle over $\CP^1$ that corresponds topologically to the trivial $S^2
\times S^2$ bundle if $k$ is even
and  to the  non-trivial one $S^2 \times_{\tau} S^2 = \CP^2 \#
\bar{\CP}^2$ if $k$ is odd. In this
correspondence, the fibers, of area $1$, are preserved and the section at infinity of
this bundle
$$
s_{\infty} = \{([z_0,z_1],[0,0,1])\}
$$
of area $\nu - \frac k2$  in $\F_{k}$ corresponds to the section  of self-intersection $-k$ that lives in $S^2
\times S^2$ if $k$ is even
and in the non-trivial $S^2$-bundle  if $k$ is
odd.   Thus it represents the class $\si_{0}-(k/2)F$ (resp.
$\si_{-1} -(\frac{k-1}{2})F$ in the non-trivial case) where $\si_{j}$ is the section  of self-intersection $j$. Therefore, the form $(\nu - \frac k2) \tau_1 +
\tau_2$ gives area $1$ to each $\CP^1$-fiber and area $\nu$ to the section of self-intersection $0$ (i.e to $B= \si_{0}$ in the trivial case, and  to $(\si_{-1} + \si_{1})/2$ in the odd case). However, our conventions for $M_{\mu}^{i}$ gives area $\mu$ to the zero section when $i=0$ and gives area $\mu$ to the section $\si_{-1}$ when $i=1$.  This means that $\nu$ must be identified with $\mu$ when $k$ is even and with $\mu + \frac 12 $ when $k$ is odd. By the
classification theorem of ruled symplectic $4$-manifolds,  this
correspondence
establishes a symplectomorphism between
$(\F_k, (\mu - \frac k2) \tau_1 +
\tau_2) $ and  $M_{\mu}^{0}$ for all even $k's$ strictly smaller than $2\mu$; similarly, there is a symplectomorphism between $(\F_{k}, (\mu + \frac 12 - \frac k2) \tau_1 +
\tau_2) $ and  $M_{\mu}^{1}$ for all odd $k's$ strictly smaller than $2\mu$.  Each such symplectomorphism endows $M_{\mu}^{i}$ with a different integrable compatible complex structure indexed by $0 \leq k < \mu$  having the section of self-intersection $-k$ holomorphically represented.  

 The element $t$  in $\pi_1(\Symp(M_\mu^i))$ is the rotation in the fibers of the Hirzebruch surface $\F_{2+i}$, once identified with $M_\mu^i$, $i=0,1$; for $i=0$, it is therefore the rotation in the fibers of $M_\mu^0 = B \times F$ round the two fixed symplectic surfaces in classes $B-F, B+F$ represented by the graph of the $\pm$ identity map from the base $B$ to the fiber $F$. Similar comments apply to $i=1$.

In the case $i=0$, the element $\widetilde t$ is the blow-up of $t$ at the point $([1,0], [0,0,1]) \in \F_2$, kept fixed under the action of $t$, identified with the center $\iota_c(0) \in S^2 \times S^2$ of the standard ball $B_c$. In the untwisted case the element $x$ is the $3$-dimensional sphere generating $\pi_3(\SO(3))$ where $\SO(3)$ is considered as acting on the first factor in the obvious way, the element $y$ corresponds to the case when $\SO(3)$ acts  on the second factor; the elements ${\widetilde x},{\widetilde y}$ are the blow-up of the $S^1$ part of that action that leaves the point  $\iota_c(0) \in S^2 \times S^2$ invariant. In the twisted case $x$ and $y$ are 3--spheres that generate the $U(2)$--K\"ahler actions on $\F_k$ when $k$ is odd. Finally, both $w$ and ${\widetilde w}$ are symplectic elements that do not correspond to K\"ahlerian actions (i.e a symplectic action preserving  an integrable complex structure compatible with the symplectic form). In the split case, recall that $\ell$ is the largest integer strictly smaller than $\mu$: if $\ell=1$ the generator $w$ is the  Samelson product of $t$ and $x$, while ${\widetilde w}$ is the Samelson product of $\widetilde t$ and ${\widetilde x}$; and if $\ell =2$, then both $w$ and ${\widetilde w}$ are higher order Samelson products. More precisely, as explained in \cite[Section 6]{AM} if  $\ell =2$ one can find commuting representatives of $t$ and $x$, so the Samelson product $[t,x]$ vanishes. Hence there is a 5--disk that bounds $[t,x]$, and the new 8--dimensional generator $w$ is a second order Samelson product made from this new disk and $x$. In general, if $\ell < \mu \leq \ell+1$, the Samelson product $[t,x,\hdots,x]$ of order $\ell-1$ vanishes, so $w$, in degree $4\ell$, is a higher order product made from a $(4\ell-3)$--disk and $x$. For the generator $\widetilde{w}$ in the blow--up manifold, there is a similar description, that is, if $\ell < \mu \leq \ell+1$ and $c \geq \mu -\ell$, the Samelson product $[\widetilde{t},\widetilde{x},\hdots,\widetilde{x}]$ of order $2\ell-2$ vanishes, so the generator $\widetilde{w}$ in degree $4\ell-2$ is a higher product  made from a $(4\ell-3)$--disk and $\widetilde{x}$. Notice that the dimension of $\widetilde{w}$ jumps by two every time $\mu$ passes an integer or $c$ passes the critical value $\ccrit=\mu -\ell$. 

There is a corresponding description for the twisted case, however, instead of considering the Samelson product of $t$ and $x$ one should consider the product of the generators $x$ and $y$ and their higher order Samelson products. 

\section{The minimal model of $\IEmbuci$}\label{se:minimal-models} 
    The goal of this section is to prove the following theorem:

\begin{thm}\label{thm:MinimalModelIEmb}
If $c \ge \ccrit$ (i.e $0 <\mu-\ell \leq c< 1$), the minimal model of $\Im \Embuci$ is 
$$\Lambda(\Im \Embuci)=(\Lambda (a,b,e,f,g,h),d_U)=\Lambda(S^2 \times S^2) \otimes \Lambda(c,g)$$ with generators in degrees $2,2,3,3,4\ell+2i-1,4\ell+2i$ and with differential
$$
d_U e = a^2, \; d_U f = b^2, \;  d_U g = d_U a = d_U b = 0, \; d_U h=q bg,
$$ 
where $\Lambda(S^2 \times S^2)$ is the minimal model for $S^2 \times S^2$ and $q$ is a non zero rational number. 
Thus the rational cohomology ring of  $\Im \Embuci$ is equal to the algebra 
$$ H^*(\Im \Embuci;\Q)=\Lambda(a,b,g,gh, \hdots, gh^n, \hdots,bh, \hdots, bh^n, \hdots)/\langle a^2,b^2,bg \rangle $$
where $n\in \N$ (see the computation of this cohomology ring in corollary \ref{cor:RationalCohomologyRingIEmb}).
 It is therefore not homotopy equivalent to a finite-dimensional CW-complex. 
\end{thm}

  \begin{proof} Any fibration $V \hookrightarrow P \to U$ for which the theory of minimal models applies (i.e. each space has a nilpotent homotopy system and the $\pi_1$ of the base acts trivially on the higher homotopy groups of the fiber) gives rise to a sequence
$$
( \Lambda(U), d_U) \longrightarrow  (\Lambda(U) \otimes \Lambda(V), d) \longrightarrow (\Lambda(V), d_V).
$$
where the differential algebra in the middle is a model for the total space. 
Let $d_{|U},d_{|V}$ represent the restriction of the differential $d$ to $U$ and $V$ respectively. The theory of minimal models implies that
$$ d_{|U} = d_U $$
$$ d_{|V} = d_V + d'$$
where $d'$ is a perturbation with image not in $\Lambda (V)$

Given the fibration 
$$ \Symp({\widetilde{M}}_{\mu,c}^i) \longrightarrow \Symp(M_{\mu}^i)  \longrightarrow \Im \Embuci \, , $$

\NI
we whish to find the model for $\Im \Embuci$. We will treat the case $i=0$ in complete detail. The case $i=1$ is completely analogous to this one. To avoid unnecessary repetitions, in the latter case we will just state the relevant propositions, leaving the proofs as exercises to the interested reader. Since the rest  of the section is mainly devoted to the case $i=0$, we will assume throughout that $i=0$, unless noted otherwise, and omit the superscript $^0$ to simplify notation.

The algebra of the minimal model of $\IEmbuc$ for $c \geq \ccrit$ follows from the computation, in Pinsonnault \cite{Pi}, of its rational homotopy groups:
$$ 
\pi_1 =0, \pi_2=\Q^2 ,  \pi_3= \Q^2 ,  \pi_{4\ell-1}=\Q, \pi_{4\ell}=\Q , \; \mbox{ and $\pi_n = 0$ for all other $n.$}
$$ 
Therefore the algebra is $\Lambda (a,b,e,f,g,h)$ where $ \deg a =\deg b = 2$, $\deg e = \deg f = 3$, $\deg g= 4\ell-1$  and $\deg h=4\ell$.
Thus we get the following fibration
$$
\Lambda (a,b,e,f,g,h), d_U \longrightarrow \Lambda (a,b,e,f,g,h) \otimes \Lambda (\widetilde t,{\widetilde x},{\widetilde y}, {\widetilde w}),d \longrightarrow \Lambda (\widetilde t,{\widetilde x},{\widetilde y}, {\widetilde w}), d_V
$$
The differential $d$ satisfies $d_{|U} = d_U $ and $ d_{|V} = d_V+ d'=d'$. So in order to find the differential $d_U$ for $\Lambda (a,b,e,f,g,h)$ it is sufficient to compute the differential $d$ for the model  $\Lambda (a,b,e,f,g,h) \otimes \Lambda (\widetilde t,{\widetilde x},{\widetilde y}, {\widetilde w})$. We need to compare this model with the minimal model  $\Lambda (t,x,y, w)$ of $\Symp(M_{\mu}^i)$ given in the last section.

\medskip
\noindent
{\it \bf Computation of the differential d}.

Let us first apply the simplest method of dimension counting. That method yields easily the following partial results:
\begin{lemma}\label{le:DifferentialPartialResults}
Without loss of generality, one may assume that the differential d satisfies:
$$
d \widetilde  t = 0, \ \ d  {\widetilde x} =a , \ \ d  {\widetilde y}=b , \ \ d  {\widetilde w} = g
$$
(and therefore the differentials of $a,b,g$ vanish). Moreover $d e$ and $d f$  must be quadratic, equal to (perhaps vanishing) linear combinations of $a^2, b^2, ab$. 
\end{lemma}

\begin{proof}
Since the middle term computes the rational cohomology of $\Symp(M_{\mu})$, we need exactly one generator of degree  1. There is no loss of generality in assuming that it is $\widetilde t$: $d \, \widetilde  t=0$.
It follows that $d \, {\widetilde x}$ and $d \, {\widetilde y}$ must be different from 0 otherwise we would have too many generators in cohomology in dimension  1. By the theory of models for fibrations,  the perturbation $d'$ has image not  in $\Lambda (\widetilde t,{\widetilde x},{\widetilde y}, {\widetilde w})$.
Therefore, without loss of generality, we may set:
$$d \, {\widetilde x} =a , \ \ d \, {\widetilde y}=b $$ which implies that $d \, a=d \, b =0$.

Now lets compute $d \, {\widetilde w}$. It does not vanish because there is no generator in the cohomology of $\Symp(M_{\mu})$  in dimension $4\ell-2$. The theory of rational models for fibrations implies that the perturbation $d'$ is dual to the boundary operator $\partial: \pi_*(B) \otimes \Q \to \pi_*(F) \otimes \Q$. Since $\partial g = {\widetilde w}$, we conclude that $d'{\widetilde w} = g$, which means that $d{\widetilde w}=g$ and implies that $dg=0$.
\end{proof}

Let us now compute the values of $d \, e $ and $d \, f$. 
As we will see below,  these will follow easily from the  computation of the Whitehead products $[a,a],$ $ [a,b], [b,b]$ in the rational homotopy of $\IEmbuc$. If the total space of the fibration $\Symp({\tMuc}) \to \Symp(M_{\mu})  \to  \Im \Emb_{\om}(c, \mu)$ were contractible, computing such  products would boil down to computing the Samelson product of corresponding elements of the fiber. But our total space is not contractible, and we have to take also into account an horizontal part in the Whitehead product.
 
   Let us briefly describe the generators of  $ \Lambda(U)$, i.e. the generators of the rational homotopy groups of $\IEmbuc$. The group $\Symp(M_{\mu})$ acts on $\IEmbuc$ by $\phi \cdot A = \mbox{image} \, (\phi |_A)$ with stabilizer equal to $\Symp(M_{\mu},B_c)$, the subgroup of symplectic diffeomorphisms which preserve (not necessarily pointwise) $B_c$, the image of the standard embedding of the ball of capacity $c$ of $\R^{2n}$ in $M_{\mu}$. This leads to the following homotopy fibration:

\begin{eqnarray}\label{FibrationSymp-IEmb}
\Symp({\tMuc}) \to \Symp(M_{\mu}) \to \IEmbuc.  
\end{eqnarray}

The elements $e,f$ and $h$ are the images by the action of $\Symp(M_{\mu})$ on $\Im \Emb_{\om}(\mu,c)$ of the elements $x,y$ and $w$ of $\pi_*(\Symp(M_{\mu})) \otimes \Q$. The elements $a,b$  are uniquely defined as those spheres in the base of that fibration whose lifts to the total space $\Symp(M_{\mu})$ are discs with boundary on the fiber equal to ${\widetilde x}$ and ${\widetilde y}$ respectively. These lifts are unique because $\pi_2(\Symp(M_{\mu})) \otimes \Q$ vanishes.  The element $g$ is defined in the following way. When $\ell >1$ it is uniquely defined as the sphere in the base of that fibration whose  lift to the total space is a disc with boundary on the fiber  equal to ${\widetilde w}$. Such a lift is unique since $\pi_{4\ell-1}(\Symp(M_{\mu})) \otimes \Q$ vanishes in this case. However, if $\ell=1$, the lift of $g$ to the total space is a class  in $ \pi_{3}(\Symp(M_{\mu}), \Symp({\tMuc})) \otimes \Q$  ($= \pi_{3}(\Im \Emb_{\om}(c, \mu)) \otimes \Q )$ which is not uniquely defined.  To make it unique, we define it by first taking the $2$-disc $D_{\widetilde x} \subset SO(3) = x \subset \Symp(M_{\mu})$ whose boundary is equal to $2 {\widetilde x}$, and then taking the commutator of $t$ and $D_{\widetilde x}$. This yields a $3$-disc $D$ lying inside $[t,x]_S = w$, whose boundary is the Samelson product $2[\widetilde t,{\widetilde x}] = 2{\widetilde w}$. Set $g = D/2 \in \pi_3(\Symp(M_{\mu}), \Symp({\tMuc})) \otimes \Q$.

\begin{lemma} The Whitehead product $[a,b]$ vanishes, and 
\[
[a,a] = e \quad [b,b] = f.
\]
\end{lemma}

   \proof{}
   
   Assume that $X,Y$ are the generators of $\pi_*(S^2) \otimes \Q$ of degrees 2 and 3, respectively, and $X_j,Y_j$ their images in the $j^{th}$ factor of $S^2 \times S^2$. We have $[X,X] = Y$ in the rational homotopy of $S^2$ and $[X_1,X_2] = 0$ 
because $[X_1,X_2]$ is the obstruction to extend the inclusion map $(S^2 \times \{ pt \}) \cup (\{ pt \} \times S^2) \to S^2 \times S^2$ to a map defined on $S^2 \times S^2$.    

Recall, from Theorem ~\ref{thm:homotopy-type-mu<2}, that $S^2 \times S^2$ is a retract of the space of embedded balls, that is, there is a a section $\sigma:S^2 \times S^2 \to \IEmb_\omega^0(c,\mu)$ of  fibration \eqref{eq:homotopy-type-mu<2}. It follows that $\sigma_*(X_1)=a,\sigma_*(X_2)=b, \sigma_*(Y_1)=e $ and $ \sigma_*(Y_2)=f$ and therefore $[a,a]=e,[b,b]=f$ and $[a,b]=0$.\qed

\bigskip
Recall that Sullivan's duality implies:
$$
d b_k =  \sum_{i,j} <b_k,[b_i, b_j]>   b_i b_j.
$$
Therefore, the last lemma implies that:
$$
d_U(e) = a^2  \quad  \mbox{and} \quad d_U(f) = b^2.
$$

 It remains to compute $dh$. 

\begin{lemma}\label{le:dh=qbg}
The differential d satisfies $dh=q bg$ where $q$ is a non zero rational number. 
\end{lemma}
\proof{}
Notice that $dh \neq 0$ if and only if
\[
\rk H^{4\ell}(\Im \Emb_{\om}(c, \mu); \Q) =
\begin{cases}
1 &\text{for $\ell=1$}\\
0 &\text{for $\ell\geq 2$}
\end{cases}  
\]
Indeed, if $dh$ did not vanish, there would be no element remaining in degree $4\ell$ when $ \ell > 1$ and there would be only one element remaining in degree 4, namely $ab$. Hence, in that case,  we would have $dh=c\tau$ where   $\tau$ is a non-zero linear combination of $a$ and $b$, since there are no closed classes  in degree $4\ell +1$ except  $c\tau$. Moreover, there is a constant $q\neq 0$ such that $dh=q bg$, because the Whitehead product $[a,g]$ vanishes. Indeed, recall that if $\ell =1$ the element $2a$ is the projection on the base $\Im \Emb_{\om}(c, \mu)$ of the $2$--disc $D_{\widetilde x} \subset D $ defined above, while $2g$ is the projection of the $3$--disc $D\subset [t,x]_S=w$. Therefore $a \subset g $ and their Whitehead product must vanish since $\pi_4(S^3) \otimes \Q=0$. If $\ell >1$ notice that the Samelson product of order $2\ell-2$, $[t,{\widetilde x},\hdots,{\widetilde x}]$, vanishes so the $(4\ell-1)$--disc 
$$ D_{\widetilde w} \equiv [t,\underbrace{{\widetilde x},\hdots,{\widetilde x}}_{2\ell-2},D_{\widetilde x}]$$ 
is well defined. Its boundary is the higher order Samelson product $2[t,{\widetilde x},\hdots,{\widetilde x}]= 2{\widetilde w}$. So $2g$ is the projection on the base of this disc.  Since $D_{\widetilde x} \subset  D_{\widetilde w} $ it follows again that $a \subset g$ and their Whitehead product must vanish because $\pi_{4\ell}(S^{4\ell-1})\otimes \Q=0$.

 We will show that $H^{4\ell}(\Im \Emb_{\om}(c, \mu); \Q)= H^{4\ell}(G_\mu/\tGuc; \Q)$ is zero-dimensional if $\ell \ge 2$ and one-dimensional if $\ell=1$, by an argument that uses the Eilenberg--Moore spectral sequence \footnote{It is obvious that $dh$ either vanishes or is equal to a non-zero multiple of $bg$. Unfortunately, one can prove that the Leray spectral sequence cannot distinguish between these two cases. We thank the referee for pointing out that the Eilengerg-Moore sequence does.}
 applied to the fibration $G_\mu/\tGuc \to B\tGuc \to BG_\mu$. This spectral sequence, which is a second quadrant spectral sequence, converges to $ H^*(G_\mu/\tGuc; \Q)$.  Its $E_2$--term is given by 
$$ E_2^{i,j}=\Tor_{H^*(BG_\mu)}^{-i,j} (\Q\, , H^*(B\tGuc)).$$ 
We follow Paul Baum's paper \cite[Section 2]{Ba} to calculate these Tor groups. Let $\Lambda$ be a graded $\Q$--algebra and,  $M$ and $N$ be $\Lambda$--modules. Then  $\Tor_\Lambda(M,N)$ is the bigraded $\Q$--module obtained as follows. Consider  a projective resolution $R$ of $M$ over $\Lambda$ given by 
$$ R = \{ \xymatrix@+.5pc{ \hdots \ \ar[r] &  \ R^{(-2)} \ \ar[r]^{f^{(-2)}} &  \ R^{(-1)} \ \ar[r]^{f^{(-1)}} & \ R^{(0)} \ \ar[r]^{f^{(0)}} & \ M \ \ar[r] & \ 0 } \}. $$
Let $L$ be the bigraded differential $\Q$--module defined by $L^{p,q}=(R^{(p)} \otimes_\Lambda N )^q$ with $d: L^{p,q} \rightarrow  L^{p+1,q} $ given by $f^{(p)}\otimes_\Lambda 1_N$. $\Tor_\Lambda(M,N)$ is the homology of $L$, that is $\Tor_\Lambda^{p,q}(M,N)=H^{p,q}(L).$

In our example we have $ \Lambda=H^*(BG_\mu;\Q)$, $M=\Q$ and $N=H^*(B\tGuc;\Q)$. The cohomology ring of $BG_\mu$ was computed by Abreu, Granja and Kitchloo in \cite{AGK}; it is given by 
$$ H^*(BG_\mu;\Q)=\Q[T,X,Y]/\langle T\prod_{i=1}^\ell(T^2+i^4X-i^2Y) \rangle \ \ \ \mbox{\rm where} \ \ \ |T|=2  \ \ \ \mbox{\rm and}  \ \ \ |X|=|Y|=4.$$
The same methods can be applied to compute the cohomology ring of $B\tGuc$. The proofs of the following two theorems are given in Appendix B.

\begin{thm}[See Theorem~\ref{thm:RationalCohomologyAlgebra}]
Let $\ell<\mu\leq\ell+1$. Then the cohomology ring of $B\tGuc$ is isomorphic to
\[
\frac{\Q[z,x,y]}{{\widetilde R}_{\mu,c}}
\]
where $z,x,y$ have degree $2$, and where the ideal ${\widetilde R}_{\mu,c}$ is given by
\[{\widetilde R}_{\mu,c}=
\begin{cases}
\langle z(z-x+y)(z-x-y)\cdots(z-\ell^{2}x+\ell y) \rangle & \text{in the case $c\geq \ccrit$,}\\
\langle z(z-x+y)(z-x-y)\cdots(z-\ell^{2}x+\ell y)(z-\ell^{2}x-\ell y) \rangle & \text{in the case $c<\ccrit$.}
\end{cases}
\]
\end{thm}

The map $B\tGuc \to BG_\mu$ induces a map in cohomology. 
\begin{thm}[See Theorem~\ref{thm:MapCohomologyAlgebras}]
The map $H^*(BG_\mu;\Q) \to H^*(B\tGuc;\Q)$ is given by
$$ \begin{array}{cll}
 T & \mapsto & z \\ 
 X & \mapsto & x^2\\ 
 Y & \mapsto & y^2+2xz.
\end{array} $$
\end{thm}
Under this map, the cohomology of $BG_\mu$ can be identified with the subring $$ H^*(BG_\mu;\Q)=\Q[z,x^2,y^2+2xz]/\langle z \prod_{i=1}^{\ell}((z-i^2x)^2-i^2y^2) \rangle.$$

We need to construct a projective resolution for $\Q$ as a $H^*(BG_\mu)$--module. We can achieve this with the augmentation of $\Lambda$, $\varepsilon: \Q[z,x^2,y^2+2xz]/\langle z \prod_{i=1}^{\ell}((z-i^2x)^2-i^2y^2) \rangle\rightarrow \Q \simeq \Lambda^0 $. Therefore we may calculate these Tor groups using the following resolution (called the Koszul resolution)  
$$\Lambda (\alpha, \beta, \gamma, \delta) \otimes \Q[z,x^2,y^2+2xz]/\langle z \prod_{i=1}^{\ell}((z-i^2x)^2-i^2y^2) \rangle,$$
with differentials given by 
\begin{eqnarray}\label{eq:DifferentialKozsul}
d(\alpha)=z, \ d(\beta)=x^2, \ d(\gamma)=y^2+2xz, \ d(\delta)= \alpha \prod_{i=1}^{\ell}((z-i^2x)^2-i^2y^2).
\end{eqnarray}
Here $\Lambda(\alpha, \beta, \gamma, \delta)$ denotes the free (bi)graded algebra on elements $\alpha, \beta, \gamma$ and $\delta$ in bidegrees $(-1,2)$, $(-1,4)$, $(-1,4)$ and $(-2,4\ell+2)$ respectively. The above complex is a module over $\Q[z,x^2,y^2+2xz]/\langle z \prod_{i=1}^{\ell}((z-i^2x)^2-i^2y^2) \rangle $, graded in external degree zero, i.e. it lies in grading $(0,*)$. It follows that the Tor groups of interest are the cohomology of the complex 
$$\Lambda (\alpha, \beta, \gamma, \delta) \otimes \Q[z,x,y]/\langle  z (z-\ell^2 x+\ell y)\prod_{i=1}^{\ell-1}((z-i^2x)^2-i^2y^2)\rangle.$$
Here we use the identification of $H^*(BG_\mu;\Q)$ as a subring of $ H^*(B\tGuc;\Q)$ and under this identification the differential of the complex above 
  satisfies the equalities (\ref{eq:DifferentialKozsul}) and $d(\eta \otimes m)= d\eta \otimes m $ with $\eta \in \Lambda (\alpha, \beta, \gamma, \delta)$ and  $m\in \Q[z,x,y]/\langle  z(z-\ell^2 x+\ell y)\prod_{i=1}^{\ell-1}((z-i^2x)^2-i^2y^2)\rangle $.
 Any class in total degree $4\ell$, which is in negative external degree, may be written as $x_{4\ell}=c_1\delta+\alpha\beta \,h_1(x,y,z)+\alpha\gamma \,h_2(x,y,z)$, where $c_1$ is a constant and $h_1$ and $h_2$ are linear combinations of classes of the type $x^{n_x}y^{n_y}z^{n_z}$ where $n_x,n_y,n_z \in \N$ such that $n_x+n_y+n_z=2\ell-2$ and therefore $|x^{n_x}y^{n_y}z^{n_z}|=4\ell -4$. For it to be closed we need 
$$c_1\alpha\prod_{i=1}^{\ell}((z-i^2x)^2-i^2y^2) +z\beta h_1+z\gamma h_2 -\alpha x^2 h_1-\alpha (y^2+2zx) h_2 =0$$
which can happen only if $c_1=0$ and all the coefficients in the linear combinations $h_1,h_2$ vanish. Hence the only closed classes are in external degree zero. Clearly, all the classes of the type $z\, x^{n_x}y^{n_y}z^{n_z}$, where $n_x,n_y,n_z \in \N$ and $n_x+n_y+n_z=2\ell-1$, are in the image of the differential $d$ because $d (\alpha \,x^{n_x}y^{n_y}z^{n_z})=z\,x^{n_x}y^{n_y}z^{n_z}$. It remains to check that all the classes of the type $x^ky^{2\ell-k}$ are also in the image of $d$ where $ 0\leq k \leq 2 \ell$, except the class $xy$ if $\ell =1$ $(k=1)$.  Note that if $k \geq 2$ then 
$$x^ky^{2\ell-k}= x^2 \,x^{k-2}y^{2\ell-k}=d (\beta \, x^{k-2}y^{2\ell-k}).$$ 
If $k=1$ and $\ell \geq 2$ then 
$$xy^{2\ell-1}=x y^{2\ell-3}\, y^2= \frac{1}{2}d(xy^{2\ell-3}(\gamma-2\alpha \, x)).$$
This shows that there are no classes in Tor in total degree $4\ell$ if $\ell>1$ and there is only one, generated by $xy$, if $\ell=1$. 
\qed
\bigskip

\begin{remark}
There is a completely analogous story for the twisted case. The methods are exactly the same and they show that only the dimension of the generators $g$ and $h$ changes.
 In this case the cohomology ring of $\BGu^1$ was  computed in \cite{AGK} where the authors showed that
 \begin{equation}\label{CohomologyAlgebraTwisted}
H^*(BG_\mu^1;\Q)=\Q[X,Y,T]/\langle \prod_{i=0}^\ell((2i+1)^2(\frac{i(i+1)}{2}(X+Y)-Y)-\frac{i^2(i+1)^2}{2}T^2) \rangle.
\end{equation}
The diffeomorphism  $\tBGuc^1 \simeq \tBG_{\mu+1-c,1-c}^0$  yields easily the cohomology ring of the blow up when  $c \geq \ccrit$:
\[
H^*(\tBGuc^1;\Q)=\Q[x,y,z]/\langle z\prod_{i=1}^\ell(z-i^2x+iy)(z-i^2x-iy) \rangle.
\]
Moreover, the map $i^{*}:H^{*}(\BGu^1;\Q)\to H^{*}(\tBGuc^1;\Q)$ is given by (see Theorem \ref{thm:MapCohomologyAlgebrasTwisted})
\[
\begin{array}{l}
X\mapsto y(y-x)+ \frac{z}{2}(7y+7z-3x)\\
Y\mapsto \frac{z}{2}(y-x+z)\\
T\mapsto 4z+2y-x.
\end{array}
\]
Note that under this map, the relation in \eqref{CohomologyAlgebraTwisted} is mapped to the product 
$$\left(-\frac12\right)^{\ell+1}z(z-(\ell+1)^2x+(\ell+1)y)\prod_{i=1}^\ell(z-i^2x+iy)(z-i^2x-iy) $$
which is a multiple of the relation in the cohomology ring of $\tBGuc^1$. 
\end{remark}

\bigskip \medskip  
\noindent
This completes the proof of Theorem~\ref{thm:MinimalModelIEmb} \end{proof}

\section{The minimal model of $\Embuci$} \label{se:rathomtypeb}

  In this section, we compute the minimal model of the space $\Embuci$ of parametrised symplectic balls. Unless noted otherwise, we assume that $i=0$ throughout and again  omit the superscript $^0$. Consider the fibration $U(2) \to \Emb_{\om}(c, \mu) \to \Im \Emb_{\om}(c, \mu)$.  First observe that this fibration is the restriction to $B_c$ of the fibration $\Symp({\tMuc}) \to \Symp(M_{\mu}) \to \Im \Emb_{\om}(c, \mu)$. This can be expressed by the following commutative diagram:

\[
\xymatrix{
 \Symp^{\id, B_c}(M_{\mu})  \ar@{^{(}->}[r]  \ar@{=}[d] & \Symp^{U(2)}(M_{\mu}, B_c) \ar[r]^{restr} \ar[d]   &  U(2) \ar[d]  \\
\Symp^{\id, B_c}(M_{\mu})   \ar@{^{(}->}[r] \ar[d]  &  \Symp(M_{\mu})   \ar[r]^{restr} \ar[d] &  \Emb_{\om}(c, \mu)  \ar[d] \\
\{B_c\}  \ar@{^{(}->}[r]  & \Im \Emb_{\om}(c, \mu)  \ar@{=}[r] & \Im \Emb_{\om}(c, \mu) 
}
\]
where $restr$ is the restriction to the standard embedded ball $B_c \subset M_{\mu}$, $\Symp^{U(2)}(M_{\mu}, B_c)$ is the subgroup of $\Symp(M_{\mu})$ formed of diffeomorphisms that preserve the ball $B_c$ and act in a $U(2)$ linear way on it, and $\Symp^{\id, B_c}(M_{\mu})$ is the subgroup of $\Symp(M_{\mu})$ formed of the elements that fix the ball $B_c$ pointwise. Recall that there is a natural homotopy equivalence between $\Symp^{U(2)}(M_{\mu}, B_c)$  and $\Symp({\tMuc})$, so the vertical fibration in the middle is equivalent to the fibration (2) of \S~1, namely $\Symp({\tMuc}) \to \Symp(M_{\mu}) \to \Im \Emb_{\om}(c, \mu)$. 

We also have the commutative diagram:

\[
\xymatrix{
U(2)  \ar[d]  & U(2)   \ar[d]  \\
\Emb_{\om}(c, \mu) \ar[d] \ar[r]^{j}  & UFr(M) \ar[d]  \\
\Im \Emb_{\om}(c, \mu) \ar[r]^{l}  &  M
}
\]
where $UFr(M)$ is the space of unitary frames of $M$, $j$ is the 1-jet map evaluated at the origin (followed by the Gram-Schmidt process assigning a unitary frame to each symplectic one), and where the last horizontal map assigns to each unparametrised ball its center (well-defined up to homotopy).  

    The minimal model for $U(2)$ is $\Lambda(u_0,v_0)$ where $deg (u_0) = 1$ and $deg(v_0) = 3$. We first show that the elements $e,f,g,h \in \pi_*(\Im \Emb_{\om}(c, \mu)) \otimes \Q$ lift to $\pi_*(\Emb_{\om}(c, \mu)) \otimes \Q$, but not $a,b$. However the difference $a-b$ does lift. On the other hand only the element $v_0$ injects in $\pi_*(\Emb_{\om}(c, \mu)) \otimes \Q$, the element $u_0$ is killed.

\begin{prop} The rational homotopy of $\Emb_{\om}(c,\mu)$ is generated, as module over $\Q$, by a single element $\widetilde h$ in dimension $4\ell$, by one element $\widetilde g$ in dimension $4\ell-1$, by three elements $v,\widetilde e,  \widetilde f$ in dimension $3$, and by a single element $\widetilde d_{a,b}$ in dimension $2$. The elements $\widetilde h, \widetilde e, \widetilde f$ are the images by the restriction map of the elements $w, x,y$ respectively. The element $v$ is the image of $v_0$, $\widetilde d_{a,b}$ is the unique lift of the difference $d_{a,b}:=a-b$ and $\widetilde g$ is the unique lift of $g$ if $\ell >1$. If $\ell =1$ the element $\widetilde g$ is well-defined up to a multiple of $v$. 
\end{prop}

\proof{} Consider the following commutative diagram of long exact sequences

\[
\xymatrix{
\ldots  \ar[r] &  \pi_k(U(2))\otimes \Q \ar[r]^-{\iota_*} \ar[d]^{\id} & \pi_k(\Emb_{\om}(c, \mu)) \otimes \Q  \ar[r]^-{\rho_*} \ar[d]^{j_*}  & \pi_k(\Im \Emb_{\om}(c, \mu)) \otimes \Q \ar[r]^-{\partial_*} \ar[d]^{l_*} & \ldots  \\
 \ldots  \ar[r]  &  \pi_k(U(2)) \otimes \Q \ar[r]  & \pi_k(SFr(M)) \otimes \Q \ar[r]   & \pi_k(M) \otimes \Q \ar[r] &   \ldots  
}
\]
Since $\pi_{4\ell}(M) \otimes \Q$ vanishes, $l_{*=4\ell}(h) = 0$, and therefore $\partial_{*}(h) =0$. Hence $\rho_{*=4\ell}$ is an isomorphism between  $\pi_{4\ell}(\Emb_{\om}(c, \mu)) \otimes \Q$ and $ \pi_{4\ell}(\Im \Emb_{\om}(c, \mu)) \otimes \Q$. Let's denote by $\widetilde h$ the lift of $h$. Since $\pi_{4\ell-1}(U(2))$ and $\pi_{4\ell-2}(U(2))$ vanish if $\ell \neq 1$, it follows that the map $\rho_{*=4\ell-1}$ is an isomorphism between $\pi_{4\ell-1}(\Emb_{\om}(c, \mu)) \otimes \Q$ and $ \pi_{4\ell-1}(\Im \Emb_{\om}(c, \mu)) \otimes \Q$. Let $\widetilde g$ be the lift of $g$. In that case, that is, if $\ell \neq 1$ and 
for $k=3$, the short sequence
$$
\pi_3(U(2)) \otimes \Q  \stackrel{\iota_*}{\to}   \pi_3(\Emb_{\om}(c, \mu)) \otimes \Q  \stackrel{\rho_*}{\to}  \pi_3(\Im \Emb_{\om}(c, \mu)) \otimes \Q 
$$
splits because $\pi_4(\Im \Emb_{\om}(c, \mu)) \otimes \Q $ and $\pi_2(U(2))$ vanish. Let's denote by $v$ the image of $v_0$ and by $\widetilde e, \widetilde f$ the lifts of $e,f$; all are well defined.

If $\ell=1$, for  $k=3$, the short sequence 
$$
\pi_3(U(2)) \otimes \Q  \stackrel{\iota_*}{\to}   \pi_3(\Emb_{\om}(c, \mu)) \otimes \Q  \stackrel{\rho_*}{\to}  \pi_3(\Im \Emb_{\om}(c, \mu)) \otimes \Q 
$$
still splits because, as we saw, $h$ is mapped to $0$, and $\pi_2(U(2))$ vanishes.
We still denote by  $v$ the image of $v_0$ and by $\widetilde e, \widetilde f, \widetilde g$ the lifts of $e,f,g$. In this case, 
 all are well-defined except $\widetilde g$ which is defined up to a multiple of the element $v$. Consider now the sequence
$$
 0 \to \pi_2(\Emb_{\om}(c, \mu)) \otimes \Q   \stackrel{\rho_*}{\to}  \pi_2(\Im \Emb_{\om}(c, \mu)) \otimes \Q \stackrel{\partial_*}{\to}  \pi_{1}(U(2)).
$$
The elements $a,b$ are by definition such that they lift to discs
$$
\phi_a, \phi_b: D^2 \to \Symp(M_{\mu}^i)
$$
with boundary equal to the elements $x,y \in \pi_1(\Symp({\tMuc})) \otimes \Q$ respectively. Therefore, their lifts to $\Emb_{\om}(c, \mu) \otimes \Q$ are  the $2$-discs
$$
\psi_a, \psi_b: D^2 \to \Emb_{\om}(c, \mu)
$$ 
defined by $\psi_{a,b}(z) = \phi_{a,b} |_{B_c}$. Hence their boundaries are the restriction of the loops  ${\widetilde x},{\widetilde y} \in \pi_1(\Symp(M_{\mu},B_c)) \otimes \Q$ to the standard ball $B_c \subset M_{\mu}$. But each of these loops preserve $B_c$ (not pointwise) and correspond to the generator of $\pi_1(U(2)) \otimes \Q$ through the identification $B^4(c) (\subset \R^4) \to B_c$.
This proves that each of $a$ and $b$ is mapped to $u_0$ by the boundary operator of the above sequence. Denote by $\widetilde d_{a,b}$ the lift to $ \pi_2(\Emb_{\om}(c, \mu)) \otimes \Q$ of the element $d_{a,b}=a-b$.

   Finally, the map $\partial_*: \pi_2(\Im \Emb_{\om}(c, \mu)) \otimes \Q \to  \pi_{1}(U(2)) \otimes \Q$ being onto, the space $\pi_1(\Emb_{\om}(c, \mu)) \otimes \Q$ must vanish. \qed

\MS
  Let's compute the minimal model of $\Emb_{\om}(c, \mu)$.  By the last proposition, a model of $\Emb_{\om}(c,\mu)$ is given by $(\Lambda(\widetilde d_{a,b}, \widetilde e, \widetilde f, \widetilde g, v, \widetilde h), d_0)$. By minimality, there is no linear term in the differential, so $d_0(\widetilde d_{a,b}) = 0$, while the constants (when $\ell >1$) in the expression 
$$
d_0 (\widetilde e) =  c_1  \widetilde d_{a,b}^{\; 2}, \quad  d_0 (\widetilde f) =  c_2 \widetilde d_{a,b}^{\; 2}  \quad d_0 v =  c_3 \widetilde d_{a,b}^{\; 2}$$
are given, by duality, by
$$
[\widetilde d_{a,b}, \widetilde d_{a,b}] = c_1 \widetilde e + c_2 \widetilde f + c_3 v.
$$
When $\ell=1$ we have to consider also $ d_0 (\widetilde g) =  c_4 \widetilde d_{a,b}^{\; 2}$.
Denoting by $\rho$ the projection $\Emb_{\om}(c, \mu) \to \Im \Emb(c,\mu)$, we have:
$$
\rho_{*}([\widetilde d_{a,b},  \widetilde d_{a,b}]) = [d_{a,b}, d_{a,b}] = [ a-b, a-b] = [a,a] + [b,b] = e + f.
$$
Therefore $c_1 = c_2 = 1$ and $c_4=0$ (if $\ell=1$), and we get $[\widetilde d_{a,b}, \widetilde d_{a,b}] = \widetilde e +  \widetilde f  + c_3 v$. Now any value of this constant leads to the same model, up to isomorphism. Indeed, since $d_0 \widetilde e = d_0 \widetilde f = \widetilde d_{a,b}^{\; 2}$ and $d_0 v = c_3 \widetilde d_{a,b}^{\; 2}$, this means that $\widetilde e$ kills $\widetilde d_{a,b}^{\; 2}$ and thus both $\widetilde f$ and $v$ can be considered as cycles (up to a reparametrization of the basis of the algebra). 

   Finally, if $\ell=1$, the differential of $\widetilde h$ is given by 
the coefficient affecting the term $\widetilde h$ in the Whitehead products $[\widetilde d_{a,b}, \widetilde e],
[\widetilde d_{a,b}, \widetilde f], [\widetilde d_{a,b}, \widetilde g], [\widetilde d_{a,b}, v]$, while if $\ell>1$ we just need to compute the Whitehead product $[\widetilde d_{a,b}, \widetilde g]$. Projecting on the base of the fibration, we see that all these coefficients must vanish, except for the coefficient $q \in \Q$ in $d_0 \widetilde h = q \widetilde d_{a,b}\widetilde g$. Indeed projecting $[\widetilde d_{a,b}, \widetilde g]$ on the base we have 
$$
\rho_{*}([\widetilde d_{a,b},  \widetilde g]) = [ d_{a,b}, g]= [a-b,g]=[a,g]-[b,g]=-qh,
$$
since  $[a,g]=0$ and the differential $d$  of the minimal model of $\Im\Emb$ satisfies $dh=qbg$ for some $q \neq 0$, as seen  in  Lemma \ref{le:dh=qbg}. This shows that the differential of $\widetilde h$ is given by 
$$ 
d_0\widetilde h=-q\widetilde d_{a,b} \widetilde g.
$$

  Denoting by $\widetilde f'$  and $v'$ the elements $\widetilde f - \widetilde e$ and $v -c_3 \widetilde e$ respectively, the sets $\{\widetilde e, \widetilde f', \widetilde g, v'\}$ and $\{\widetilde e, \widetilde f', v'\}$ form  a basis of the $3$-dimensional generators for the cases $\ell=1$ and $\ell>1$ respectively. These same methods apply  also to the computation of the minimal model of $\Emb^1_{\om}{(c, \mu)}$, that is, to the twisted case. So this proves the following:
\begin{thm}\label{thm:MinimalModelEmb}
 If $0<\mu-\ell \leq c <1$, a minimal model of $\Emb^i_{\om}{(c, \mu)}$ is given by 
$$\Lambda(\Emb^i_{\om}{(c, \mu)})=(\Lambda(\widetilde d_{a,b}, \widetilde e, \widetilde f', v',\widetilde g, \widetilde h), d_0)$$
 with generators of degrees $2,3,3,3,4\ell+2i-1,4\ell+2i$ and with differential defined  by 
$$ d_0 \widetilde d_{a,b}  = d_0 \widetilde f' = d_0 \widetilde g = d_0 v'  = 0, \  d_0 \widetilde e =   \widetilde d_{a,b}^{\; 2}
 \mbox{ \  and \ } d_0\widetilde h=-q\widetilde d_{a,b} \widetilde g $$ where $q$ is a non zero rational number. Then the rational cohomology ring of $\Emb^i_{\om}{(c, \mu)}$ is given by 
$$ H^*(\Emb^i_{\om}{(c, \mu)}; \Q)=\Lambda(\widetilde d_{a,b},\widetilde f',\widetilde g, v', \widetilde g \widetilde h, \hdots,\widetilde g \widetilde h^n, \hdots, \widetilde d_{a,b}\widetilde h, \hdots, \widetilde d_{a,b}\widetilde h^n, \hdots)/\langle \widetilde d_{a,b}^{\;2},\widetilde d_{a,b}\widetilde g \rangle $$ where $n \in \N$  (see the computation of this cohomology ring in corollary \ref{cor:RationalCohomologyEmb}).
\end{thm}
\BS

\section{Cohomology rings } \label{se:cohomology}

It is easy  to describe the cohomology ring of $\Im \Emb^i_\omega(c,\mu)$ with rational coefficients.  A careful comparation between the Serre spectral sequence of the fibration 
\begin{equation}\label{eq:FibrationIEmbUntwisted}
S^{4\ell+2i-1} \times \Omega S^{4\ell+2i+1} \ {\longrightarrow} \ G_\mu^i/{\tGuc^i} \ {\longrightarrow} \ M^i_\mu
\end{equation}
(recall that ${\tGuc^1} \simeq {\widetilde G}^0_{\mu+1-c,1-c}$) and Theorem \ref{thm:MinimalModelIEmb} gives the  cohomology ring of $\Im \Emb^i_\omega(c,\mu)$.

\begin{cor}\label{cor:RationalCohomologyRingIEmb}
If $0<\mu-\ell\leq c< 1 $ the cohomology ring of $\Im \Emb^i_\omega(c,\mu)$ with rational coefficients is given by 
$$ H^*(\Im \Emb^i_\omega(c,\mu); \Q)= \Lambda(a,b,g,gh, \hdots, gh^n, \hdots,bh, \hdots, bh^n, \hdots)/ \langle a^2,b^2,bg \rangle, $$ 
that is, 
$$ H^*(\Im \Emb^i_\omega(c,\mu); \Q)= H^*(S^2 \times S^2;\Q) \otimes \Lambda(g,gh, \hdots, gh^n, \hdots,bh, \hdots, bh^n, \hdots)/ \langle bg \rangle, $$
where $n \in \N$, $b$ is a generator of $H^2(S^2 \times S^2;\Q)$, and $g,h$ correspond to the generators of the cohomology ring  $H^* \left( S^{4\ell+2i-1} \times \Omega S^{4\ell+2i+1}; \Q \right) $ where  $|g|=4\ell+2i-1$ and $|h|=4\ell+2i$.
\end{cor} 
\proof{}
We give the proof for the untwisted case. The case $i=1$ is analogous to this one; we leave its proof to the reader. 
The rational cohomology ring of the fiber is given by 
$$H^*(S^{4\ell-1} \times \Omega S^{4\ell+1}; \Q)= \Lambda(g) \otimes \Q[h].$$ 
 
We showed in the proof of Lemma \ref{le:dh=qbg} that $H^4(\Im \Emb^0_\omega(c,\mu),\Q)$ is one dimensional and that $\rk H^{4\ell}(\Im \Emb^0_\omega(c,\mu),\Q)=0$ when $\ell>1$. This implies that in the $E_2$--term of the Serre spectral sequence of the fibration \eqref{eq:FibrationIEmbUntwisted} the differential $d_2h$ does not vanish. This is clear when $\ell >1$ for dimensional reasons. When $\ell =1$  if we had $d_2h=0$ then $h$ would survive to the $E_\infty$ page of the spectral sequence and, unless $d_4g=ab$, we would have two generators in the cohomology group $H^4(\Im \Emb^0_\omega(c,\mu); \Q)$, namely $h$ and $ab$. However, it follows from the minimal model computation that $\rk H^3(\Im \Emb^0_\omega(c,\mu); \Q)=1$ which implies that $g$ is a permanent cycle and therefore $d_rg=0$ for all $r \geq 2$. We can assume that  $d_2h=bg$. Then  the generators $bh^n$ and $gh^n$ where $n \in \N$ survive to the $E_\infty$ page  of the spectral sequence. For all these generators, except for $gh^n$ when $\ell=1$ and $ n\in \N$, this follows simply for dimensional reasons since $E_{p,q}=0$ for all $p\geq 5$ and $q \geq 0$. When $\ell=1$ one knows from  the computation of the minimal model of $\Im \Emb^0_\omega(c,\mu)$ that  $\rk H^{4n+3}(\Im \Emb^0_\omega(c,\mu); \Q)=1$, and it is easy to verify that $gh^n$, for each $n$,  is the single element in dimension $4n+3$ that can survive to the $E_\infty$--page of the spectral sequence. This completes the proof. \qed

\medskip

A comparation of Theorem \ref{thm:MinimalModelEmb} and the Serre spectral sequence of the fibration
\begin{equation}\label{eq:FibrationU(2)-Emb-IEmb}
 U(2) \longrightarrow  \Emb^i_\omega(c,\mu) \longrightarrow \Im \Emb^i_\omega(c,\mu)
\end{equation}  
yields the cohomology ring of $\Emb^i_\omega(c,\mu)$ with rational coefficients.
\begin{cor} \label{cor:RationalCohomologyEmb} If $0<\mu-\ell\leq c< 1 $ then 
$$H^*( \Emb^i_\omega(c,\mu); \Q)\cong \Lambda(b, f, v,g, gh, \hdots, gh^n, \hdots,bh, \hdots, bh^n, \hdots)/\langle b^2,bg \rangle$$ where $H^*(U(2);\Q) \cong \Lambda(u,v)$, $|f|=3$ and  $a,b,g,gh^n,bh^n$ with $n\in \N$ correspond to  the generators of the cohomology ring of $\Im \Emb^i_\omega(c,\mu)$. \end{cor}
\proof From the computation of the minimal model in Theorem \ref{thm:MinimalModelEmb}, it follows that there is no generator in degree 1 in the cohomology ring of $ \Emb^i_\omega(c,\mu)$ so, in the $E_2$--page of the Serre spectral sequence of the fibration \eqref{eq:FibrationU(2)-Emb-IEmb}, the differential satisfies $d_2u \neq 0$.
 Therefore $d_2u$ is a linear combination of $a$ and $b$. Notice that  the minimal model computation also shows that there is no element in degree 4 in the cohomology ring. Hence the element $ab$ in the $E_2$--page has to be in the image of $d_2$ or $d_4$. The computation of the minimal model implies that we need to have two generators of degree 3 in the cohomology ring for all cases except when $\ell=1$ and $i=0$ (in this latter case it has three generators). Hence $v$ is a permanent cycle and we can choose $d_2u=a$. Then one has $d_2ub=ab$ as desired and 
the element $ua$ survives to the $E_\infty$ page. The element $ua$ correponds to the generator $f$. The element $g$ survives to the $E_\infty$--page since it is the only candidate that could represent the generator in dimension $4\ell+2i-1$ that exists by the minimal model computation. It is not hard to see that the generators $gh^n$ also survive  to the  $E_\infty$ page and they correspond to the generators $\widetilde g \widetilde h^n$ in the minimal model. Finally we see that the generators $bh^n$ cannot be in the image of $d_r$ with $r \geq 2$, so they also survive to the $E_\infty$ page. Moreover they correspond to the elements $s_n=\widetilde h^{n-1}(\widetilde h \widetilde d_{a,b} + nq \widetilde e \widetilde g)$ in the minimal model, where $n \in \N$, which clearly satisfy $d_0 s_n=0$.  
\qed

\begin{remark}
Notice that this cohomology ring is equivalent to the one given in Theorem \ref{thm:MinimalModelEmb}. Indeed the difference between the two is that, here, we use the generators of the cohomology ring of $\Im \Emb^i_\omega(c,\mu)$ to describe the ring while, there, we used the generators of the minimal model.
\end{remark}

\subsection{The split case with $1< \mu \leq 2$}

 Recall from section \ref{se:ParticularCase} that if $\mu$ lies in this interval we have  the following fibration
\begin{equation}
\label{seq}
\xymatrix{ \Omega\Sigma^2 \SO(3)/\Omega S^3  \ar[r]^-{\bar{\imath}} & G_\mu^0 /\tGuco \ar[r]^-{\bar{\pi}} & {S^2 \times  S^2}}.
\end{equation}
 One can compute the cohomology ring of the space $\Im \Emb_\omega^0(c,\mu)$ with $\Z_p$  coefficients and $p$ prime, using  this fibration. Let $\Gamma_{\Z_p}[x]$ denote  the divided polynomial algebra on the generator $x$. This is, by definition, the $\Z_p$--algebra with basis $x_0=1,x_1,x_2,\hdots$ and multiplication given by 
\[
x_ix_j= \binom{i+j}{i} x_1^{i+j}
\]
As one can check, there is an isomorphism 
\[
\Gamma_{\Z_p}[x] \approx \Z_p[x_1,x_p,x_{p^2},\hdots]/\langle x_1^p,x_p^p,x_{p^2}^p,\hdots\rangle =\bigotimes_{i\geq 0}\Z_p[x_{p^i}]/\langle x_{p^i}^p\rangle
\]
\begin{cor} \label{co:crzp}If $0<\mu-1\leq c< 1 $ and $p \neq 2$ then 
$$ H^*(\Im \Emb_\omega^0(c,\mu); \Z_p)= \Lambda(a,b,g)/\langle a^2,b^2,bg \rangle \otimes g\Gamma_{\Z_p}[h] \otimes b\Gamma_{\Z_p}[h]$$
where $|a |= |b |= 2$, $|g|=3$, $|h|=4$ and $\tau\Gamma_{\Z_p}[h]$, with $\tau = g$ or $\tau = b$, stands for the infinitely generated algebra in which $\tau$ commutes with every element.
\end{cor}
\proof{}
First notice that the fiber $\Omega \Sigma^2 \SO(3) /\Omega  S^3$ is equivalent to the space $S^3 \times \Omega S^5$  away from the prime 2. Therefore we get  $$H^*( \Omega \Sigma^2 \SO(3) /\Omega  S^3; \Z_p)= \Lambda(g) \otimes \Gamma_{\Z_p}[h],$$ where $p\neq 2$, $|g| = 3$ and $|h| = 4$. The same argument as the one in the proof of the Lemma \ref{le:dh=qbg}, using the Eilenberg--Moore spectral sequence, shows that $H^4(\Im \Emb_\omega^0(c,\mu); \Z_p)$ is one dimensional if $p\neq 2$.  Since $\rk H^{4n+3}(\Im \Emb_\omega^0(c,\mu); \Q)=1$ where $n \in \N_0$ it follows that $ H^{4n+3}(\Im \Emb_\omega^0(c,\mu); \Z_p)$ is at least one dimensional. Then using again the Serre spectral sequence of the fibration \eqref{seq}
and an argument similar to the one used in  corollary  \ref{cor:RationalCohomologyRingIEmb} we obtain the desired result. \qed

Next, we will see that $\Im \Emb^0$ has $\Z_2$--torsion and therefore  the cohomology ring with these  coefficients is not as simple to describe as  the previous ones.
\begin{cor} \label{co:crz2}  When   $0<\mu-1\leq c< 1 $, the cohomology groups with $\Z_2$ coefficients of the space $\Im \Emb_\omega^0(c,\mu)$  are given by 
\begin{equation}\label{isoh} H^*(\Im \Emb_\omega^0(c,\mu); \Z_2) =  H^*(S^2\times S^2; \Z_2) \otimes  H^*(\Omega \Sigma^2 \SO(3)/\Omega S^3;\Z_2)
\end{equation}
(as vector spaces). Moreover, as an algebra  $$H_*(\Omega \Sigma^2 \SO(3)/\Omega S^3;\Z_2)=T(w_2,w_3,w_4) \otimes_{T(w_2)}\Z_2$$ where $T$ denotes the tensor algebra, that is, the free noncommutative algebra on the generators $w_i$ with degrees  $|w_i|=i$. Therefore 
the cohomology ring of $\Im \Emb_\omega^0(c,\mu)$ with $\Z_2$ coefficients is given by 
$$ H^*(\Im \Emb_\omega^0(c,\mu); \Z_2) \cong   H^*(S^2\times S^2; \Z_2) \otimes  A$$ where $A$ has an infinite number of generators.
\end{cor}
\proof{}
Since the inclusions $\Omega  S^3 \hookrightarrow \Omega \Sigma^2 \SO(3), \tGuco \hookrightarrow G_\mu^0$ and $S^1 \times S^1 \times S^1 \hookrightarrow S^1 \times \SO(3) \times \SO(3)$ induce injective maps in homology with $\Z_2$ coefficients, it follows from the Leray--Hirsch Theorem that we have the following isomorphisms as vector spaces
$$H^*( \Omega \Sigma^2 \SO(3);\Z_2) \cong H^*( \Omega \Sigma^2 \SO(3) /\Omega  S^3; \Z_2) \otimes H^*(\Omega  S^3; \Z_2), $$
$$ H^*(G_\mu^0 ; \Z_2) \cong H^*(G_\mu^0/{\tGuco} ; \Z_2)\otimes H^*(\tGuco ; \Z_2) \ \ \ {\mbox{and}} $$
$$H^*(S^1 \times \SO(3) \times \SO(3); \Z_2) \cong H^*(S^2\times S^2; \Z_2) \otimes H^*(S^1 \times S^1 \times S^1; \Z_2).$$
Moreover, since the fibrations $\pi:G_\mu^0 \rightarrow S^1 \times \SO(3) \times \SO(3)$ and $\widetilde{\pi}:\tGuco \rightarrow S^1 \times S^1 \times S^1$ are (weakly) homotopically trivial, we obtain the following isomorphisms as graded algebras 
$$  H^*(G_\mu^0 ; \Z_2) \cong H^*(S^1 \times \SO(3) \times \SO(3); \Z_2)\otimes H^*( \Omega \Sigma^2 \SO(3);\Z_2), $$
$$H^*(\tGuco ; \Z_2) \cong H^*(S^1 \times S^1 \times S^1; \Z_2) \otimes H^*(\Omega  S^3; \Z_2). $$
The five previous isomorphisms yield the isomorphisms  
\begin{equation}\label{ve:iso2}
H^*(G_\mu^0/\tGuco ; \Z_2) \cong H^*(\Omega \Sigma^2 \SO(3)/\Omega S^3;\Z_2) \otimes  H^*(S^2\times S^2; \Z_2)
\end{equation}
 as vector spaces, that is to say \eqref{isoh}.

It follows that the homomorphism   ${\bar{\imath}}^*$ is surjective  and the Serre spectral sequence of the fibration \eqref{seq} collapses at $E_2$.
Therefore
$$E_\infty^{*,*} \cong E_2^{*,*} \cong H^*(\Omega \Sigma^2 \SO(3)/\Omega S^3;\Z_2) \otimes  H^*(S^2\times S^2; \Z_2)$$ 
as bigraded modules. But this does not directly shows that the isomorphism (\ref{ve:iso2})  also holds as a graded algebra isomorphism. However, it is clear that $H^*(G_\mu^0/\tGuco ; \Z_2)$  has a subalgebra ${\bar{\pi}}^*(H^*(S^2\times S^2; \Z_2)) \cong H^*(S^2\times S^2; \Z_2)$. 
Although it is not easy to describe the $\Z_2$--cohomology of the space $\Omega \Sigma^2 \SO(3) /\Omega  S^3$,  one can calculate its $\Z_2$--homology. For this,  recall that the homology of $\Omega \Sigma X$, with its Pontrjagyn product, is the free tensor algebra on the homology of $X$ for any connected space $X$. Hence the map $\Omega  S^3 \to \Omega \Sigma^2 \SO(3)$  corresponds to the obvious inclusion of tensor algebras over $\Z_2$:
$$ T(w_2) \to T(w_2,w_3,w_4)$$ where $T$ denotes the tensor algebra and  $|w_i|=i$. The Bar spectral sequence for a principal fibration can then be applied to give 
$$ H_*( \Omega \Sigma^2 \SO(3) /\Omega  S^3; \Z_2)= T(w_2,w_3,w_4) \otimes_{T(w_2)}\Z_2,$$
as Hopf algebras, where $T(w_2)$ acts  by product on the left factor and by mapping $w_2$ to the zero map on the right factor.
In \cite{KLW} the reader will find the necessary results on the Bar spectral sequence (cf. Theorem 4.2) and further references. 
By a simple counting argument, since this is a non-commutative algebra,  the (graded-commutative)  cohomology ring  $H^*(\Omega \Sigma^2 \SO(3) /\Omega  S^3; \Z_2)$ must have an infinite number of generators. From the $E_2$ page of the spectral sequence of the fibration \eqref{seq} we can then conclude that 
$$H^*(G_\mu^0/\tGuco ; \Z_2) \cong   H^*(S^2\times S^2; \Z_2) \otimes A, $$
 as graded algebras, where $A$ has an infinite number of generators, but it is not necessarily isomorphic as a graded  algebra to $H^*(\Omega \Sigma^2 \SO(3)/\Omega S^3;\Z_2)$. This isomorphism completes the proof. 
\qed

\medskip

There is a similar picture for the cohomology ring of $\Emb_\omega^0(c,\mu)$ with $\Z_p$ coefficients and $p$ prime. 

\begin{cor} If $0<\mu-1\leq c< 1 $ and $p\neq 2$ then 
$$H^*( \Emb_\omega^0(c,\mu); \Z_p)\cong \Lambda(b,f,g,v)/\langle b^2,bg \rangle \otimes  g \,\Gamma_{\Z_p}[h]\otimes  b\Gamma_{\Z_p}[h],$$
where $\tau\Gamma_{\Z_p}[h]$, with $\tau = g$ or $\tau = b$, stands for the infinitely generated algebra $$\Z_p[\tau h_1,\hdots, \tau h_1^{p-1},\tau h_p,\hdots, \tau h_p^{p-1},\tau h_{p^2}, \hdots,\tau h_{p^2}^{p-1}, \hdots ]$$ where the generators $h_i$ are the generators of the divided polynomial algebra $\Gamma_{\Z_p}[h]$.
\end{cor}

 Using this corollary and the Serre spectral sequence of fibration \ref{eq:FibrationU(2)-Emb-IEmb}, we get
\begin{cor}
If $0<\mu-1\leq c< 1 $ then 
$$H^*( \Emb_\omega^0(c,\mu); \Z_2)\cong \Lambda(b,f,v)/\langle b^2 \rangle  \otimes A'$$
where the algebra $A'$  has an infinite number of generators.
\end{cor}
\bigskip
\NI

\appendix
\section{Integrable complex structures and homotopy decomposition of $B\tGuci$} \label{se:appendixA}

\subsection{Spaces of compatible integrable complex structures}   

Let $(M,\om)$ be a symplectic $4$-manifold. Denote by $\jj_{\om}$ the space of compatible almost complex structures and by $\Ii_{\om}$ the subset of integrable ones. Given an integrable $J$, we write $H^{0,q}_{J}(M)$ for the $q^{\text{th}}$ Dolbeault cohomology group with coefficients in the sheaf of germs of holomorphic functions, and $H^{0,q}_{J}(TM)$ for the $q^{\text{th}}$ Dolbeault cohomology group with coefficients in the sheaf of germs of holomorphic vector fields.

In their paper~\cite{AGK}, Abreu-Granja-Kitchloo prove that, under some cohomological conditions, $\Ii_{\om}$ is a genuine Fr\'echet submanifold of $\jj_{\om}$ whose tangent bundle may be described using standard deformation theory, namely
\begin{thm}[\cite{AGK} Theorem 2.3]\label{thmA}
Let $(M,\om)$ be a symplectic $4$-manifold, and let $J\in\Ii_{\om}$. If the cohomology groups $H^{0,2}_{J}(M)$ and $H^{0,2}_{J}(TM)$ are zero, then $\Ii_{\om}$ is a submanifold of $\jj_{\om}$ near $J$.
Moreover, the tangent space of $\Ii_{\om}$ at $J$ is naturally identified with   
$T_{J}((\Diff(M) \cdot J)\cap \Ii_{\om})\oplus H^{0,1}_{J}(TM)$. Here, $H^{0,1}_{J}(TM)$ represents the
the moduli space of infinitesimal compatible deformations of $J$ in $\Ii_{\om}$ that coincides with the moduli space of infinitesimal deformations of $J$ in the set of all integrable structures $\Ii$.
\end{thm}

The actions of various natural subgroups of $\Diff(M)$ on $\Ii$ give rise to different partitions of $\Ii_{\om}$. Let $\Diff_{[\om]}$ denote the group of diffeomorphisms of $M$ preserving the cohomology class $[\om]$ and write $\Hol_{[\om]}(J)$ for the subgroup of complex automorphisms of $(M,J)$. Let $\Iso(\om,J)$ denote the K\"ahler isometry group of $(M,\om,J)$. The next result shows that in some cases the part of the $\Diff_{[\om]}$-orbit of $J$ that lies in $\Ii_{\om}$ may be identified with the $\Symp(M,\om)$-orbit:

\begin{thm}[\cite{AGK} Corollary 2.6]\label{thmB}
If $J\in\Ii_{\om}$ is such that the inclusion $\Iso(\om,J)\into\Hol_{[\om]}(J)$ is a weak homotopy equivalence, then the inclusion of the $\Symp(M,\om)$-orbit of $J$ in $(\Diff_{[\om]}\cdot J)\cap\Ii_{\om}$ 
\[
\Symp(M,\om)/\Iso(\om,J)\into (\Diff_{[\om]}\cdot J)\cap\Ii_{\om}
\] 
is also a weak homotopy equivalence.
\end{thm}

Recall that the natural stratifications of $\jj(\Mui)$ and $\jj(\tMuci)$ described in Section~\ref{se:decompostion} are defined geometrically in terms of $J$-holomorphic curves in certain homology classes. The restriction of each stratum to integrable compatible complex structures can be understood as follows. Given a symplectic manifold $(M,\om)$, let us denote by $\Mm(A,\jj_{\om})$ the space of pairs $(u,J)\in C^{\infty}(\CP^{1},M)\times \jj_{\om}$ such that $u:\CP^{1}\to M$ is a somewhere injective $J$-holomorphic map whose image represents the homology class $A$. This space is always a smooth manifold whose image $U_{A}$ under the projection $\pi:\Mm(A,\jj_{\om})\to\jj_{\om}$ is the set of all $J$ such that $A$ is represented by an irreducible $J$-holomorphic sphere. The next proposition gives conditions ensuring that the stratum $U_{A}$ is tranversal to $\Ii_{\om}$ and that its normal bundle at $J\in\Ii_{\om}$ may be described in terms of deformation theory.

\begin{thm}[\cite{AGK} Theorem 2.9]\label{thmC}
Let $(M,\om,J)$ be a K\"ahler $4$-manifold such that the cohomology groups $H^{0,2}_{J}(M)$ and $H^{0,2}_{J}(TM)$ are zero. Suppose that $(u,J)\in\Mm(A,\jj_{\om})$ is such that $u^{*}:H^{0,1}_{J}(TM)\to H^{0,1}(u^{*}(TM))$ is an isomorphism. Then the projection $\pi:\Mm(A,\jj_{\om})\to\jj_{\om}$ is tranversal at $(u,J)$ to $\Ii_{\om}\subset\jj_{\om}$ and the infinitesimal complement to the image $U_{A}$ of $\pi$ at $J$ can be identified with the moduli space of infinitesimal deformations $H^{0,1}_{J}(TM)$
\end{thm}

\subsection{The case of rational ruled surfaces} 
It is classical that given any $k \ge 0$ and a compatible integrable complex (Hirzebruch) structure $J_{k}$ in the $k^{\text{th}}$ stratum $\jj_{\mu,k}^{i}$, both $H^{0,2}_{J_{k}}(\Mui)$ and $H^{0,2}_{J_{k}}(T\Mui)$ are zero. Moreover,  for any $J\in\Ii^{i}_{\mu,k} := \jj_{\mu,k}^{i} \cap \Ii_{\mu}^i$ and a $J$-holomorphic map $u:\CP^{1}\to(\Mui,J)\simeq\F_{k}$ representing a section of negative self-intersection, the induced map $u^{*}:H^{0,1}_{J}(TM)\to H^{0,1}(u^{*}(TM))$ is an isomorphism. By theorems~\ref{thmA} and~\ref{thmC}, it follows that the restriction of the stratification of $\jjui$ to the integrable complex structures $\Ii_{\mu}^{i}$ defines a stratification of $\Ii_{\mu}^{i}$, and that the strata in both stratifications have normal slices isomorphic to moduli spaces of infinitesimal deformations. Now, it is also well known that $\Hol(J_{k})$ retracts on $\Iso(\om,J_{k})=K_{k}$. Consequently, by theorem~\ref{thmB}, the symplectic orbit $\Gui\cdot J_{k}$ is weakly homotopy equivalent to the intersection $(\Diff_{[\om]}\cdot J)\cap\Ii_{\om}$. Finally, because any two integrable complex structures belonging to the same stratum are in the same Teichm\"uller class (that is, there exists a $\phi\in\Diff_{0}$ sending one to the other), each stratum is weakly homotopy equivalent to a symplectic orbit:
\[
\Ii_{\mu,k}^{i} = (\Diff_{0}\cdot J_{k})\cap\Ii_{\mu}\simeq \Gui/K_{k}
\]  
Therefore, the actions of $\Gui$ on the spaces $\jjui$ and $\Ii_{\mu}^{i}$ are homotopically equivalent. In particular, the total space $\Ii_{\mu}^{i}$ is weakly contractible and one may compute the homotopy type of $\Gui$ using the homotopy pushout diagram
\begin{equation}
\label{EquivalentDiagram}
\xymatrix{
\Gui\times_{K_{\ell}} S^{4\ell+2i-3}  \ar[r] \ar[d] & \Ii_{\mu,\ell}^{i} \simeq G_{\mu}^{i}/ K_{\ell} \ar[d] \\
\Ii^{i}_{\mu}- \Ii_{\mu,\ell}^{i} \ar[r] &  \Ii_\mu^{i}
}
\end{equation}
where as usual $\ell < \mu \le \ell +1$. As explained in~\cite{AGK}, the action of $K_{k}$ on the normal slice $H^{0,1}_{J_{k}}(T\Mui)$ may be determined either directly (i.e. by looking at the action of $K_{k}$ on \v{C}ech cocycles associated to an open cover) or by applying the Atiyah-Bott fixed points formula to compute the character of the virtual representation of $K_{k}$ on the equivariant elliptic complex
\[
\Omega^{0,0}_{J_{k}}(T\Mui)\to\Omega^{0,1}_{J_{k}}(T\Mui)\to\Omega^{0,2}_{J_{k}}(T\Mui).
\]
This gives
\begin{thm}[\cite{AGK} Theorem 4.1] If $k$ is even, the representation of $K_{k}$ on $H^{0,1}_{J_{k}}(T\Mui)$ is given by 
$\det\otimes\sym^{k-2}(\C^{2})$, where $\det$ represents the standard representation of $S^{1}$, and $\sym^{k-2}(\C^{2})$ denotes the $k-2$ symmetric power of the canonical representation of $\SO(3)=\SU(2)/\pm 1$ on $\C^{2}$. Similarly, if $k$ is odd, we get the representation 
$\det^{-k} \otimes \sym^{k-2}(\C^{2})$, where $\det$ denotes the determinant representation of $\U(2)$ and $\sym^{k-2}(\C^{2})$ denotes the $k-2$ symmetric power of the canonical representation of $\U(2)$ on $\C^{2}$. 
\end{thm}
In particular, the representation of $K_{k}$ on $H^{0,1}_{J_{k}}(T\Mui)$ is independent of $\mu>k$. Consequently, an easy induction argument over the number of strata shows that 
\[
EG_{\mu}^{i}\times_{G_{\mu}^{i}} (\Ii_{\mu}^{i} - \Ii_{\mu,\ell}^{i}) \simeq BG_{\mu-1}^{i}.
\]

\subsection{Blow-ups of rational ruled surfaces}
 
We now show that similar results hold in the case of the symplectic blow-ups $\tMuci$, for all values of $\mu$ and $c\in(0,1)$. In particular, we no longer assume that $c \ge \ccrit$.

\begin{lemma}\label{le:Vanishing}
For any $J\in\Ii_{\mu,c}^{i}$, the cohomology groups $H^{0,2}_{J}(\tMuci)$ and $H^{0,2}_{J}(T\tMuci)$ are zero. 
\end{lemma}
\begin{proof}
(See also~\cite{Ko} \S 5.2(a)(iv) p.220.) Without loss of generality we can suppose that $i=0$. Because the exceptional class $E$ is symplectically indecomposable, it is represented by an embedded $J$-holomorphic exceptional sphere that can be blowed down (in the complex category). This shows that any compatible complex structure on $\tMuco$ is obtained by blowing-up a Hirzebruch structure on $\Muo$. Now, for any K\"ahler manifold $(X,J)$ of complex dimension $n$, the Hodge numbers $h^{p,q}:=\rk H^{p,q}_{J}(X)$ satisfy $h^{n-p,n-q}=h^{p,q}=h^{q,p}$. Since the Hodge numbers $h^{p,0}$ are birational invariants, it follows that
\[
\rk H_{J}^{0,2}(\tMuco)=h^{2,0}(\tMuco)=h^{2,0}(\Muo)=0.
\]
As for $H^{0,2}_{J}(T\tMuco):=\check{H}^{2}(T\tMuco)$, Serre duality implies that $\check{H}^{2}(T\tMuco)^{\vee}\simeq \check{H}^{0}(\mathcal{K}\otimes\Omega^{1})$. Now, $\tMuco$ contains a real two-dimensional family of embedded rational curves of zero self-intersection (the fibers) which cover a dense open set, and the restriction of the rank $2$ bundle $\mathcal{K}\otimes\Omega^{1}$ to any of those curves is isomorphic to $\Oo(-4)\oplus\Oo(-2)$. Hence, $\mathcal{K}\otimes\Omega^{1}$ cannot have a nontrivial holomorphic section. 
\end{proof}
\begin{lemma}\label{le:Teichmuller}
Each stratum $\Ii_{\mu,c,k}^{i}$ is covered by the Teichm\"uller orbit, that is, for any pair $\tJ_{0},\tJ_{1}\in\Ii_{\mu,c,k}^{i}$, there exists a $\phi\in\Diff_{0}(M)$ such that $\tJ_{1}=\phi_{*}\tJ_{0}$.
\end{lemma}
\begin{proof}
Given a pair $\tJ_{0},\tJ_{1}\in\Ii_{\mu,c,k}^{i}$, the class $E$ is represented by $\tJ_{i}$-holomorphic exceptional curves $\Sigma_{i}$ that are symplectically isotopic. So, we may assume that $\Sigma_{0}=\Sigma_{1}$. By blowing down $(M,\Sigma,\tJ_{i})$, we get two complex structures on the same underlying marked $4$-manifold $(M,p,J_{i})$. The unmarked complex surfaces are both isotopic to the Hirzebruch surface $\F_{k}$. Note that any such isotopy sends $p$ to the zero section $s_{0}$ of $\F_{k}$. The statement follows from the fact that the identity component of the complex automorphism group of $\F_{k}$ (which is isomorphic to the semi-direct product $\PSL(2;\C)\ltimes(\C^{*}\times H^{0}(\CP^{1};\Oo(k))$ ) acts transitively on $s_{0}$, see~\cite{DG}.
\end{proof}

\begin{lemma}
The inclusion $\Iso(\om,\tJ_{k})\into\Hol_{[\om]}(\tJ_{k})$ is a weak homotopy equivalence.
\end{lemma}
\begin{proof}
This follows from the corresponding statement for the Hirzebruch surfaces.
\end{proof}

\begin{lemma}\label{le:Transversality}
Let $\widetilde{\F}_{m=2k}\simeq \widetilde{\STS}$ denote the blow-up of the Hirzebruch surface $\F_{m}$ at some point $p$. Let $C$ be the unique embedded rational curve representing either $B-kF$ or $B-kF-E$ (depending on whether $p$ lies outside the zero section of $\F_m$ or not), and denote by $u:C\to\widetilde{\F}_{m}$ its inclusion. Then the induced map $u^{*}:H^{0,1}(T\widetilde{\F}_{m})\to H^{0,1}(u^{*}(T\widetilde{\F}_{m}))$ is an isomorphism.
\end{lemma}
\begin{proof}
Consider the exact sequence of sheaves associated to the inclusion $u:C\to\widetilde{\F}_{m}$
\[
0\to \Oo(-C)\to \Oo_{\widetilde{\F}_{m}}\to \Oo_{C}\to 0
\]
Tensoring with $T\widetilde{\F}_{m}$ we get the short exact sequence
\[
0\to \Oo(-C)\otimes T\widetilde{\F}_{m}\to T\widetilde{\F}_{m}\to \Oo_{C}\otimes T\widetilde{\F}_{m}\to 0
\]
whose associated cohomology sequence is
\[
\cdots\to H^{1}(\widetilde{\F}_{m};T\widetilde{\F}_{m})\stackrel{u^{*}}{\to} H^{1}(\mathbb{P}^{1};u^{*}T\widetilde{\F}_{m})\to H^{2}(\widetilde{\F}_{m};\Oo(-C)\otimes T\widetilde{\F}_{m})\to\cdots
\]
The sheaf $\Oo(-C)\otimes T\widetilde{\F}_{m}$ being locally free, Serre duality implies that $H^{2}(\widetilde{\F}_{m};\Oo(-C)\otimes T\widetilde{\F}_{m})\simeq H^{0}(\widetilde{\F}_{m};\Oo(C)\otimes T\widetilde{\F}_{m}^{\vee}\otimes K_{T\widetilde{\F}_{m}})$. But,  since $C\cdot F = 1$, the restriction of $\Oo(C)\otimes T\widetilde{\F}_{m}^{\vee}\otimes K_{T\widetilde{\F}_{m}}$ to any fiber $F$ is isomorphic to $\Oo(1)\otimes(\Oo(-4)\oplus\Oo(-2))$. It follows that $\Oo(C)\otimes T\widetilde{\F}_{m}^{\vee}\otimes K_{T\widetilde{\F}_{m}}$ has no nontrivial sections and, by duality, that $H^{2}(\widetilde{\F}_{m};\Oo(-C)\otimes T\widetilde{\F}_{m})=0$.
\end{proof}
\begin{cor}
The action of $\Symp(\tMuci)$ on $\tjjuci$ is homotopy equivalent to its restriction to $\widetilde{\Ii}_{\mu,c}^{i}$.
\end{cor} 
\begin{cor}\label{cor:Contractibility}
The space $\widetilde{\Ii}_{\mu,c}^{i}$ of compatible integrable complex structures of $\tMuci$ is contractible.
\end{cor}

\section{Algebraic computations} \label{se:appendixB}

\subsection{Conventions} 
In order to carry the computation of the cohomology ring of $B\tGuci$, we follow the conventions used in~\cite{AGK} and~\cite{AM}:
\begin{enumerate}
\item Let $T^{4}\subset\U(4)$ act in the standard way on $\C^{4}$. Given an integer $n\geq 0$, the action of the subtorus $T^{2}_{n}:=(ns+t,t,s,s)$ is Hamiltonian with moment map
\[(z_{1},\ldots,z_{4})\mapsto(n|z_{1}|^{2}+|z_{3}|^{2}+|z_{4}|^{2},|z_{1}|^{2}+|z_{2}|^{2})\]
We identify $M_{\mu}^{0}=(S^{2}\times S^{2},\mu\sigma\oplus\sigma)$ with each of the toric Hirzebruch surface $\F_{2k}^{\mu}$, $0\leq k\leq \ell$ (where as usual $\ell < \mu \le \ell+1$), defined as the symplectic quotient $\C^{4}/\!/T^{2}_{2k}$ at the regular value $(\mu+k,1)$ endowed with the residual action of the torus $T(2k):=(0,u,v,0)\subset T^{4}$. The image $\Delta(2k)$ of the moment map $\phi_{2k}$ is the convex hull of
\[
\{(0,0),(1,0),(1,\mu+k),(0,\mu-k)\}
\]
Similarly, we identify $\Mul=(\NTB,\om_{\mu})$ with the toric Hirzebruch surface $\F_{2k-1}^{\mu}$, $1\leq k\leq \ell$, defined as the symplectic quotient $\C^{4}/\!/T^{2}_{2k-1}$ at the value $(\mu+k,1)$. The moment polygon $\Delta(2k-1):=\phi_{2k-1}(\F_{2k-1})$ of the residual action of the $2$-torus $(0,u,v,0)$ is the convex hull of 
\[
\{(0,0),(1,0),(1,\mu+k),(0,\mu-k+1)\}
\]
Note that the group $\Symp_{h}(M_{\mu})$ of symplectomorphisms acting trivially on homology being connected, any two identifications of $\F_{n}^{\mu}$ with $M_{\mu}^{i}$ are isotopic and lead to isotopic identifications of $\Symp_{h}(\F_{n}^{\mu})$ with $\Symp(M_{\mu}^{i})$. 

\item The K\"ahler isometry group of $\F_{n}^{\mu}$ is $N(T^{2}_{n})/T^{2}_{n}$ where $N(T^{2}_{n})$ is the normalizer of $T^{2}_{n}$ in $\U(4)$. There is a natural isomorphism\footnote{In the untwisted case, we assume $\mu>1$ so that the permutation of the two $S^{2}$ factors is not an isometry.} $N(T_{0})/T^{2}_{0}\simeq \SO(3)\times \SO(3):=K(0)$, while for $k\geq 1$, we have $N(T^{2}_{2k})/T_{2k}\simeq S^{1}\times \SO(3):=K(2k)$ and $N(T^{2}_{2k-1})/T^{2}_{2k-1}\simeq \U(2):=K(2k-1)$. The restrictions of these isomorphisms to the maximal tori are given in coordinates by
\begin{align*}
(u,v)&\mapsto (-u,v)\in T(0):=S^{1}\times S^{1}\subset K(0)\\
(u,v)&\mapsto (u,ku+v)\in T(2k):=S^{1}\times S^{1}\subset K(2k)\\
(u,v)&\mapsto (u+v,ku+(k-1)v)\in T(2k-1):=S^{1}\times S^{1}\subset K(2k-1)
\end{align*}
These identifications imply that the moment polygon associated to the maximal torus $T(n)=S^{1}\times S^{1}\subset K(n)$ is the image of $\Delta(n)$ under the transformation $C_{n}\in\GL(2,\Z)$ given by
\[
\begin{matrix}
 C_{0}=\begin{pmatrix}-1 & 0 \\ 0 & 1   \end{pmatrix}
&~C_{2k}=\begin{pmatrix} 1 & 0 \\ -k & 1   \end{pmatrix}
&~C_{2k-1}=\begin{pmatrix} 1-k & 1 \\ k & -1 \end{pmatrix}
\end{matrix}
\]

\item We identify the symplectic blow-up $\tMuco$ with the equivariant blow-up of the Hirzebruch surfaces $\F_{n}^{\mu}$ for appropriate parameter $\mu$ and capacity $c\in(0,1)$.
\begin{enumerate}
\item We define the even torus action $\tT(2k)$ as the equivariant blow-up of the toric action of $T(2k)$ on $\F_{2k}^{\mu}$ at the fixed point $(0,0)$ with capacity $c$. 
\item The odd torus action $\tT(2k-1)$, $k\geq 1$, is obtained by blowing up the toric action of $T(2k-1)$ on $\F_{2k-1}^{\mu-c}$ at the fixed point $(0,0)$ with capacity $1-c$.
\end{enumerate}
Our choices imply that under the blow-down map, $\tT(n)$ is sent to the maximal torus of $K(n)$, for all $n\geq 0$. Again, because $\Symp(\tMuco)$ is connected (see~\cite{LP,Pi}), all choices involved in these identifications give the same maps up to homotopy. Note also that when $c<\ccrit := \mu - \ell$, $\tMuco$ admits exactly $2\ell+1$ inequivalent toric structures $\tT(0),\ldots, \tT(2\ell)$, while when $c \ge \ccrit$, it admits only $2\ell$ of those, namely $\tT(0),\ldots, \tT(2\ell-1)$. Note that the free variable $n$ indexing these objects corresponds to the free variable $j$ indexing the strata in \S~2.4.
\item The cohomology ring of $B\tT(n)$ is isomorphic to $\Q[x_{n},y_{n}]$ where $|x_{n}|=|y_{n}|=2$. We identify the generators $x_{n}$, $y_{n}$ with the cohomology classes induced by the circle actions whose moment maps are, respectively, the first and the second component of the moment map associated to $\tT(n)$. Geometrically, $y_{n}$ is induced by the lift to $\tMuco$ of a rotation of the base of $M^{0}_{\mu}$, while $x_{n}$ is induced by a rotation of the fibers. Note that since we work only with topological groups up to rational equivalences, we will also denote by $\{x_{n}, y_{n}\}$ the generators in $\pi_{1}\tT(n)$ and in $\pi_{2}B\tT(n)$.
\end{enumerate}

\subsection{The isotropy representation of $\tT(n)$}

\begin{prop}
The character of the representation of $\tT(n)$, $n\geq 1$, on $H^{0,1}(T\tMuci)$ is given by
\[
\widetilde{\kappa}(n)=
\begin{cases}
x(y^{k-1}+\cdots+y^{-(k-1)}) &\text{if $n=2k-1$,}\\
x(y^{k-1}+\cdots+y^{-k}) &\text{if $n=2k$.}
\end{cases}
\]
Consequently, the equivariant Euler class of the representation is
\[
\pm\widetilde{e}_{n}=
\begin{cases}
(x+(k-1)y)(x+(k-2)y)\cdots(x+(1-k)y)&\text{if $n=2k-1$,}\\
(x+(k-1)y)(x+(k-2)y)\cdots(x-ky)&\text{if $n=2k$.}
\end{cases}
\]
In particular, $\widetilde{e}_{n}\in H^{*}(B\tT(n);A)$ is nonzero for any coefficient ring $A$.
\end{prop}
\begin{proof}
Following~\cite{AGK}, we compute the character of the virtual representation of the group $\widetilde{H}(n)$ of holomorphic automorphisms of $\tMuci$ on the equivariant elliptic complex
\[
0\to\Omega^{0,0}(T\tMuci)\to\Omega^{0,1}(T\tMuci)\to\Omega^{0,2}(T\tMuci)\to 0
\] 
using the Atiyah-Bott fixed point formula applied to the maximal torus $\tT(n)\subset \widetilde{H}(n)$. Note that since $H^{0,2}(T\tMuci)=0$, the index computes the character of $H^{0,0}(T\tMuci)-H^{0,1}(T\tMuci)$. Note also that the action of $\widetilde{H}(n)$ on $H^{0,0}(T\tMuci)\simeq \mathrm{Lie}(\widetilde{H}(n))$ is isomorphic to the adjoint representation.

In the even case $n=2k\geq 2$, the isotropy weights $(w_{1},w_{2})$ of the toric action at the $5$ fixed points $\{p_{i}\}$ are
\[
\begin{array}{c|c|c|c|c|c}
        & p_{1} & p_{2}      & p_{3}  & p_{4}     & p_{5} \\
\hline
w_{1}   & y^{k}/x & 1/y      & y^{k}x & y         & y^{k+1}/x \\
\hline
w_{2}   & y       & 1/y^{k}x & 1/y    & x/y^{k+1} & x/y^{k}
\end{array}
\]
Writing $\sigma_{i}$ for the elementary symmetric polynomials in two variables, the index of the virtual representation is given by
\begin{align}
\widetilde{I}(2k)=&\sum_{p_{i}}\frac{\sigma_{1}(w_{1},w_{2})\sigma_{2}(w_{1},w_{2})}{(1-w_{1})(1-w_{2})}\\
=& \frac{1}{xy^{k}}\big( y^{2k}+\cdots+y+1 +xy^{k-1}(2y+1)\big) - x\big(y^{k-1}+\cdots+y^{-k}\big).
\end{align}
Since the number of negative terms, $2k$, is equal to the (complex) dimension of $H^{0,1}(T\tMuci)$, it follows that the character of the isotropy representation is given by the negative part of the above formula, that is,
\[\widetilde{\kappa}(2k)=x(y^{k-1}+\cdots+y^{-k}),\] 
Similarly, in the odd case $n=2k-1\geq 1$, the isotropy weights $(w_{1},w_{2})$ of the toric action at the $5$ fixed points $\{p_{i}\}$ are
\[
\begin{array}{c|c|c|c|c|c}
& p_{1} & p_{2} & p_{3} & p_{4} & p_{5} \\
\hline
w_{1} & x^{k}/y^{k+1} & x^{k-1}/y^{k} & y/x           & x^{k}/y^{k-1} & x/y \\
\hline
w_{2} & y^{k}/x^{k-1} & x/y           & y^{k-1}/x^{k} & y/x           & y^{k+1}/x^{k} \\
\end{array}
\]
and the index computation shows that 
the character of the representation $H^{0,1}(T\tMuci)$ is 
\[\widetilde{\kappa}(2k-1)=y\left(\left(\frac{x}{y}\right)^{k-1}+\cdots+\left(\frac{x}{y}\right)^{1-k}\right).\] 
In both cases, the computation of the equivariant Euler class follows by naturality.
\end{proof}

From the Gysin exact sequence of the fibration $\big( S^{2n-1}\big)_{h\tT(n)}\to B\tT(n)$, we immediately obtain 

\begin{cor}
The rational cohomology of $\big( S^{2n-1}\big)_{h\tT(n)} $ is isomorphic to $\Q[x_{n},y_{n}]/\langle \widetilde{e}_{n}\rangle$.
\end{cor}

\subsection{The cohomology module $H^{*}(B\tGuco;\Z)$} 
The homotopy decomposition of $B\tGuco$ in Theorem ~\ref{thm:HomotopyPushoutBlowUp} yields a pullback diagram

\begin{equation}\label{DiagramCohomology}
\xymatrix{
H^{*}\left(\big(S^{2m-1}\big)_{h\tT(m)};A\right)  & \ar[l]^-{\pi^{*}} H^{*}(B\tT(m);A)   \\
H^{*}(B\tG_{\mu',c}^{0};A) \ar[u]^{j_{m}^{*}}   & \ar[l]^{} H^{*}(B\tG_{\mu,c}^{0};A) \ar[u]^{i_{m}^{*}}   \\
 }
\end{equation}
where $m$ is the index of the last stratum of $\tjjuco$, and where
\[
\mu'=
\begin{cases}
\ell +c & \text{if $c<\ccrit$}\\
\ell   & \text{if $c \ge \ccrit$}
\end{cases}
\]
(here $\ell$ is the lower integral part of $\mu$) so that $\tG_{\mu',c}^{0}$ is the group associated with a stratification having one less stratum. Because the map $\pi^{*}$ is surjective, the associated Mayer-Vietoris sequence splits into the short exact sequence
\[
0\to H^{*}(B\tG_{\mu,c}^{0};A)\to H^{*}(B\tG_{\mu',c}^{0};A)\oplus H^{*}(B\tT(m);A)\to H^{*}\left(\left(S^{2m-1}\right)_{h\tT(m)};A\right)\to 0
\]
which reduces to the short exact sequence
\[
0\to \langle \widetilde{e}_{m}\rangle = \Sigma^{\deg(\widetilde{e}_{m})} H^{*}(B\tT(m);A)\to H^{*}(B\tG_{\mu,c}^{0};A)\to H^{*}(B\tG_{\mu',c}^{0};A)\to 0
\]
where $\Sigma^{n}$ stands for the $n$-fold suspension of graded abelian groups. Both sequences split over any field coefficients and, because all their terms are finitely generated, it follows that they also split over $\Z$. For $\ell=0$ we have $ H^{*}(B\tG_{0,c}^{0};\Z)\simeq H^{*}(B\tT(0);\Z)$ and, by induction, we get

\begin{thm}\label{CohomologyModuleBlowUp} As a module,
\[
H^{*}(B\tG_{\mu,c}^{0};\Z)\simeq \bigoplus_{i=0}^{m}\Sigma^{2i} H^{*}(B \tilde{T}^{2};\Z)
\text{~~where~} 
m:=
\begin{cases}
2\ell-1 & \text{if $c \ge \ccrit$}\\
2\ell   & \text{if $c<\ccrit$}
\end{cases}
\]
In particular, $ H^{*}(B\tG_{\mu,c}^{0};\Z)$ is torsion free.
\end{thm}

\subsection{The rational cohomology ring $H^{*}(B\tGuco;\Q)$}

We know from Theorem~\ref{InflationBlowUp} that the map $\tpsi_{\mu,c}:B\tGuco\to\BFDiff_{*}$ induces a surjection in rational cohomology, and we know from Section~\ref{subsection:untwisted} that there is a rational homotopy equivalence $\BFDiff_{*}\simeq B(S^{1}\times S^{1}\times S^{1})$. It follows that $H^{*}(B\tGuco;\Q)\simeq \Q[x,y,z]/\widetilde{R}_{\mu,c}$ where $\widetilde{R}_{\mu,c}$ is the kernel of $\tpsi_{\mu,c}^{*}$, and where the three generators $x$, $y$, $z$ are of degree $2$. Now, the homotopy decomposition of $B\tGuco$ yields, at the rational cohomology level, an extended pullback diagram
\begin{equation}\label{ExtendedDiagramCohomology}
\xymatrix{
\Q[x_{m},y_{m}]/\langle \widetilde{e}_{m} \rangle   &\ar[l]^-{\pi^*} \Q[x_{m},y_{m}]   &\\
\Q[x,y,z]/\widetilde{R}_{\mu',c} \ar[u]^{j_{m}^{*}}  & \ar[l]^{} \ar[u]^{i_{m}^{*}} \Q[x,y,z]/\widetilde{R}_{\mu,c}     & \\
  &  & \ar[ul]_-{\tpsi_{\mu,c}^*} \ar[ull]^{\tpsi_{\mu',c}^*}  \ar[uul]_{\tpsi_{m}^*} \Q[x,y,z] }
\end{equation}

So, in order to compute the ideal $\widetilde{R}_{\mu,c}$, one has to understand the maps $\tpsis_{n}:\Q[x,y,z]\to\Q[x_{n},y_{n}]$ for all $n\geq 0$. 
For that, it is enough to consider the relations in $\pi_{1}\FDiff_{*}$ between the generators of $\pi_{1}\tT(n)$. We first observe that when $\mu>1$, the maps $\tT(n)\to\tGuco\to\FDiff_{*}$ induce injective maps of fundamental groups. 
Then, from the classification of Hamiltonian $T^{2}$-actions and Hamiltonian $S^{1}$-actions on $4$-manifolds, it is easy to see that

\begin{lemma}\label{RelationsInTori}
In $\pi_{1}(\tGuco)\simeq\pi_{1}\FDiff_{*}$, for all admissible values $k,k' \geq 1$, we have the identifications
\begin{align*}
y_{2k}                  &= kx_{0}+y_{0}\\
k'x_{2k}-y_{2k}          &= kx_{2k'}-y_{2k'}\\
kx_{2k}+y_{2k}          &= (k+1)x_{2k-1}+ky_{2k-1}\\
(k-1)x_{2k'-1}+ky_{2k'-1} &= (l-1)x_{2k-1}+k'y_{2k-1}\\
x_{1}                   &= y_{0}-x_{0}
\end{align*} 
\end{lemma}

\begin{prop}
Let $\tpsis_{n}:H^{*}(\BFDiff_{*};\Q)\to H^{*}(B\tT(n);\Q)\simeq\Q[x_{n},y_{n}]$ be the map induced in cohomology by the inclusion $B\tilde{T}(n)\to\BFDiff_{*}$. Given $w\in H^{*}(\BFDiff_{*};\Q)$ corresponding to any element of the fundamental group, define $(a_{n},b_{n})$ by setting $\tpsis_{n}(w)=a_{n}x_{n}+b_{n}y_{n}$. Then, for $k\geq 1$, we have 
\[
\begin{array}{ll}
a_{2k}=ka_{2}+(k-1)a_{0}~ & ~b_{2k}=ka_{0}+b_{0}\\ 
a_{2k-1}=-ka_{0}+b_{0}~   & ~b_{2k-1}=ka_{2}+(2k+1)a_{0}-b_{0}
\end{array}
\]
which  shows that the coefficients $\{a_{0},b_{0},a_{2}\}$ determine $w$ and $\tpsis_{n}(w)$ for all $n\geq0$.
\end{prop}
\begin{proof}
Let $(s,t):S^{1}\to T^{2}$ stands for the inclusion $\theta\mapsto (s\theta,t\theta)$. Each relation in lemma~\ref{RelationsInTori} gives rise to a relation in cohomology by looking at the induced commutative square. For instance, the first two relations yield
\begin{equation*}
\begin{matrix}
\xymatrix{
  \Q[s] & \ar[l]^-{(0,1)^{*}} \Q[x_{2k},y_{2k}]   \\
\Q[x_{0},y_{0}] \ar[u]^{(k,1)^{*}}   & \ar[l]^{\tpsis_{0}} \Q[x,y,z] \ar[u]^{\tpsis_{2k}}   \\
 }
&
\mkern 15mu
&
\xymatrix{
  \Q[s] & \ar[l]^-{(k,-1)^{*}} \Q[x_{2k'},y_{2k'}]   \\
\Q[x_{2k},y_{2k}] \ar[u]^{(k',-1)^{*}}   & \ar[l]^{\tpsis_{2k}} \Q[x,y,z] \ar[u]^{\tpsis_{2k'}}   \\
 }
\end{matrix}
\end{equation*}
from which we immediately get $b_{2k}=ka_{0}+b_{0}$ for $k\geq 0$. Setting $k'=k+1$ in the second relation, we obtain the recursive formula $a_{2k}=ka_{2}+(k-1)a_{0}$, for $k\geq 1$. In the same way, one gets a recursive formula for the coefficients $(a_{2k-1},b_{2k-1})$, $n\geq 1$, by setting $k'=k+1$ in the forth relation. Then, one obtains explicit formulae for all coefficients, in terms of $\{a_{0}, b_{0}, a_{2}\}$ only, by using the remaining two relations. 
\end{proof}

We can define an explicit isomorphism $H^{*}(\BFDiff_{*};\Q)\simeq\Q[x,y,z]$ by choosing $x,y,z$ as the elements corresponding to the parameters $\{a_{0}=0,b_{0}=1,a_{2}=0\}$, $\{a_{0}=-1,b_{0}=0,a_{2}=1\}$,  and $\{a_{0}=0,b_{0}=0,a_{2}=1\}$ respectively. 

\begin{cor} Let $\{x,y,z\}$ be the generators of $H^{*}(\BFDiff_{*};\Q)$ defined above. Then the maps $\tpsis_{n}$ are given by the formulae
\[\begin{matrix}
\tpsis_{0}(x)=y_{0} & \tpsis_{0}(y)=-x_{0} & \tpsis_{0}(z)=0
\end{matrix}
\]
and, for $k\geq1$,
\[
\begin{array}{lll}
\tpsis_{2k}(x)=y_{2k} &\mkern 15mu & \tpsis_{2k-1}(x)=x_{2k-1}-y_{2k-1} \\
\tpsis_{2k}(y)=x_{2k}-ky_{2k}         &\mkern 15mu & \tpsis_{2k-1}(y)=kx_{2k-1}-(k+1)y_{2k-1} \\
\tpsis_{2k}(z)=kx_{2k}        &\mkern 15mu & \tpsis_{2k-1}(z)=ky_{2k-1}  
\end{array}
\]
Their kernels are the ideals
\begin{align*}
k_{2k}   &=\langle z-k^{2}x-ky \rangle \\
k_{2k-1} &=\langle z-k^{2}x+ky \rangle
\end{align*}

\end{cor}

\begin{thm}\label{thm:RationalCohomologyAlgebra}
Given $\mu\geq 1$ and $c\in(0,1)$, the rational cohomology ring of $\tBGuco$ is isomorphic to
\[
\frac{\Q[x,y,z]}{\tR_{\mu,c}}
\]
where $x,y,z$ have degree $2$, and where the ideal $\tR_{\mu,c}$ is given by
\[\tR_{\mu,c}=
\begin{cases}
\langle z(z-x+y)(z-x-y)\cdots(z-\ell^{2}x+\ell y)\rangle & \text{in the case $c\geq\ccrit$,}\\
\langle z(z-x+y)(z-x-y)\cdots(z-\ell^{2}x+\ell y)(z-\ell^{2}x-\ell y)\rangle & \text{in the case $c<\ccrit$.}
\end{cases}
\]
\end{thm}
\begin{proof}
The proof is by induction on the number $m$ of strata in $\tjjuco$. Fix some $c\in(0,1)$. Then the case $m=1$ corresponds to $\mu=1$ (hence $\ell=0$, $\ccrit=1$). But it is proved in~\cite{LP} that $\tG_{1,c}^{0}$ retracts onto $\tT(0)$, so the result holds in that case. Now assume the statement is true for some $\mu'>1$ for which there are $m-1$ strata and consider a $\mu>\mu'$ for which $\tjjuco$ contains $m$ strata. The diagram~\eqref{ExtendedDiagramCohomology} implies that $\tR_{\mu,c}\subset\tR_{\mu',c}\cap k_{m}$. Since $\tR_{\mu',c}$ and $k_{m}$ are coprime, it follows that $\tR_{\mu',c}\cap k_{m}=\tR_{\mu',c}\cdot k_{m}$. By Theorem~\ref{CohomologyModuleBlowUp}, the first relation in $H^{*}(B\tGuco;\Q)$ must occur in degree $2m$ so that $\tR_{\mu,c}=\tR_{\mu',c}\cdot k_{m}$. The statement follows.
\end{proof}

\subsection{The map $H^{*}(\BGu^i;\Q)\to H^{*}(\tBGuc^i;\Q)$}
The map $\BGu^{0}\to\BFDiff$ defined in Theorem~\ref{Inflation} induces a surjective map in rational cohomology. The rational equivalence $\BFDiff\simeq_{\Q}B(S^{1}\times \SO(3)\times \SO(3))$ allows one to choose generators of $H^{*}(\BFDiff;\Q)$: let  $T$ be the generator of degree $2$ corresponding to the $S^{1}$ factor, and denote by $X$ and $Y$ the two generators of degree $4$ corresponding to the two $\SO(3)$ factors. Then, the rational cohomology of $\BGu^{0}$ is given by
\[
H^*(BG_\mu^0;\Q)=\Q[X,Y,T]/\langle T\prod_{i=1}^\ell(T^2+i^4X-i^2Y) \rangle, \ \ \  \text{where $|T|=2$ and $|X|=|Y|=4$.}
\]

\begin{thm}\label{thm:MapCohomologyAlgebras}
The map $i^{*}:H^{*}(\BGu^0;\Q)\to H^{*}(\tBGuc^0;\Q)$ is given by
\[
\begin{array}{l}
X\mapsto x^{2}\\
Y\mapsto y^{2}+2xz\\
T\mapsto z.
\end{array}
\]
\end{thm}
\begin{proof}
Let denote by $A_{2k}$ and $X_{2k}$, $k\geq 1$, the generators of degree $2$ and $4$ in the rational cohomology of $BK(2k)\simeq B(S^{1}\times\SO(3))$. When $k\geq 1$, the torus $\tT(2k)$ maps to the maximal torus $S^{1}\times S^{1}\subset K(2k)\simeq S^{1}\times \SO(3)$. At the cohomology level, it follows that we have
\[
\begin{matrix}
A_{2k} &\to & x_{2k} \\
X_{2k} &\to & y_{2k}^{2}
\end{matrix}
\]
where $A_{2k}$ and $X_{2k}$ are the generators of $H^{2}(B(S^{1}\times\SO(3)))$ and $H^{4}(B(S^{1}\times\SO(3)))$. Let consider the diagram
\[
\xymatrix{
H^{*}(B\tT(2k)) &   H^{*}(B\FDiff_{*}) \ar[l]_-{\tpsis_{2k}} \\
H^{*}(BK(2k)) \ar[u]  & H^{*}(B\FDiff) \ar[l]_-{\psi^{*}_{2k}} \ar[u]^{\psi^{*}}
}
\]
It was shown in~\cite{AGK} that the map $\psi^{*}_{2k}$ verifies
\[
\begin{array}{lll}
\psi^{*}_{2k}(T) & = & kA_{2k}\\
\psi^{*}_{2k}(X) & = & X_{2k}\\
\psi^{*}_{2k}(Y) & = & A_{2k}^{2}+k^{2}X_{2k}.
\end{array}
\]
Consequently, for the diagram to commute we must have
\[
\begin{array}{lll}
\psi^{*}(T) & = & z\\
\psi^{*}(X) & = & x^{2}\\
\psi^{*}(Y) & = & y^{2}+2xz
\end{array}
\]
modulo elements in $\ker(\tpsis_{2k})=(z-k^{2}x-ky)$. But since this must hold for all $k\geq 1$, and since $\cap_{k} \ker(\tpsis_{2k})=\emptyset$, we see that $\psi^{*}$ is indeed given by the formulae above.
\end{proof}

In the twisted case the rational equivalence $H^{*}(\BFDiff;\Q) \simeq_\Q B(S^1 \times \SU(2) \times \SU(2))$ still gives generators $T,X$ and $Y$ such that $|T|=2$ and $|X|=|Y|=4$ corresponding now to $S^1$ and to the $\SU(2)$ factors. The rational cohomology ring of $BG_\mu^1$ was computed in \cite{AGK} and it is given by 

\[
H^*(BG_\mu^1;\Q)=\Q[X,Y,T]/\langle \prod_{i=0}^\ell((2i+1)^2(\frac{i(i+1)}{2}(X+Y)-Y)-\frac{i^2(i+1)^2}{2}T^2) \rangle.
\]
Since $\tBGuc^1 \simeq \tBG_{\mu+1-c,1-c}^0$ one gets immediately, from Theorem \ref{thm:RationalCohomologyAlgebra}, the cohomology ring in the twisted case, when $\ell < \mu \leq \ell +1$ and  $c \geq \ccrit$:

\[
H^*(\tBGuc^1;\Q)=\Q[x,y,z]/\langle z\prod_{i=1}^\ell(z-i^2x+iy)(z-i^2x-iy) \rangle.
\]

\begin{thm}\label{thm:MapCohomologyAlgebrasTwisted}
The map $i^{*}:H^{*}(\BGu^1;\Q)\to H^{*}(\tBGuc^1;\Q)$ is given by
\[
\begin{array}{l}
X\mapsto y(y-x) + \frac{z}{2}(7y+7z-3x)\\
Y\mapsto \frac{z}{2}(y-x+z)\\
T\mapsto 4z+2y-x
\end{array}
\]
\end{thm}
\begin{proof}
The proof goes exactly as in the corresponding theorem in the split case. We just need to note that now at the cohomology level we have 
\[
\begin{array}{l}
A_{2k-1}  \to x_{2k-1}+y_{2k-1} \\
X_{2k-1}  \to x_{2k-1}y_{2k-1}
\end{array}
\]
where $A_{2k-1}$ and $X_{2k-1}$ are the generators of $H^{*}(BU(2))$, 
and to recall from \cite{AGK} that the map $\psi^*_{2k-1}$ verifies 
\[
\begin{array}{lll}
\psi^{*}_{2k-1}(T) & = & (2k-1)A_{2k-1}\\
\psi^{*}_{2k-1}(X) & = & k(k-1)A_{2k-1}^2+\frac{2+k-k^2}{2}X_{2k-1}\\
\psi^{*}_{2k-1}(Y) & = & \frac{k(k-1)}{2}X_{2k-1}.
\end{array}
\]
\end{proof}


\end{document}